\documentclass[12pt, twoside]{article}
\usepackage{amsmath,amsthm,amssymb}
\usepackage{times}
\usepackage{enumerate}
\usepackage{eucal}
\usepackage[all]{xy}
\def\titlerunning#1{\gdef\titrun{#1}}
\makeatletter
\def\author#1{\gdef\autrun{\def\and{\unskip, }#1}\gdef\@author{#1}}
\def\address#1{{\def\and{\\\hspace*{18pt}}\renewcommand{\thefootnote}{}%
\footnote {#1}}%
\markboth{\autrun}{\titrun}}
\makeatother
\def\email#1{e-mail: #1}
\def\subjclass#1{{\renewcommand{\thefootnote}{}%
\footnote{\emph{Mathematics Subject Classification (2010):} #1}}}

\newtheorem{teo}{Theorem}[section]
\newtheorem{cor}[teo]{Corollary}
\newtheorem{lemma}[teo]{Lemma}
\newtheorem{prop}[teo]{Proposition}
\theoremstyle{definition}
\newtheorem{defo}[teo]{Definition}
\newtheorem{remark}[teo]{Remark}

\newtheorem{ejem}[teo]{Example}
\numberwithin{equation}{subsection}

\newcommand{\mono}{\hookrightarrow}

\newcommand{\ra}{\rightarrow}

\newcommand{\epi}{\mbox{$\to$\hspace{-0.35cm}$\to$}}

\def\umono{\ar@{_{(}->}[u]}
\def\uumono{\ar@{_{(}->}[uu]}

\def\lmono{\ar@{_{(}->}[l]}
\def\llmono{\ar@{_{(}->}[ll]}

\newcommand{\limdir}{\operatornamewithlimits{\hbox{$\varinjlim$}}}

\newcommand{\holim}{\operatornamewithlimits{holim}}
\newcommand{\hocolim}{\operatornamewithlimits{hocolim}}

\newcommand{\Hom}{{\rm Hom \,}}
\newcommand{\map}{{\rm map \,}}
\newcommand{\Rep}{{\rm Rep \,}}

\newcommand{\Z}{{\mathbb Z}}
\newcommand{\F}{{\mathbb F}}
\newcommand{\Q}{{\mathbb Q}}

\CompileMatrices
\SelectTips{cm}{}

\begin{document}

\baselineskip=17pt

\titlerunning{Homotopy idempotent functors on classifying spaces}
\title{Homotopy idempotent functors on classifying spaces}
\author{Nat\`alia Castellana and Ram\'on Flores}

\maketitle
\address{N. Castellana: Departamento de Matem\'aticas, Universidad Aut\'onoma de Barcelona, 08193
Bellaterra, Spain; \email{natalia@mat.uab.es}
\and
R. Flores: Departamento de Estad{\'\i}stica, Universidad Carlos III, 28029 Colmenarejo, Spain; \email{rflores@est-econ.uc3m.es} (corresponding author)}
\subjclass{Primary 55P20; Secondary 55P80}
\begin{abstract}
Fix a prime $p$. Since their definition in the context of Localization Theory, the homotopy functors $P_{B\Z/p}$ and $CW_{B\Z/p}$ have shown to be powerful tools to understand and describe the mod $p$ structure of a space. In this paper, we study the effect of these functors on a wide class of spaces which includes classifying spaces of compact Lie groups and their homotopical analogues. Moreover, we investigate their relationship in this context with other relevant functors in the analysis of the mod $p$
homotopy, such as Bousfield-Kan completion and Bousfield homological localization.
\end{abstract}
%%% ----------------------------------------------------------------------

%%% ----------------------------------------------------------------------\Large
\section{Introduction}

Let $A$ and $X$ be two connected topological spaces. The study of
the homotopical properties of $X$ that are visible through the
mapping space $\map(A,X)$ is called the $A$-homotopy
theory of $X$ and was proposed by E. Dror-Farjoun in \cite{Dror-Farjoun95}. In this
context, it is particularly important to describe the behaviour of
the nullification $P_{\Sigma^iA}$ and the cellularization
${CW}_{\Sigma^iA}$ (see definitions in Section 2), which are
functors that play in $A$-homotopy theory the same role as the
connected covers and Postnikov pieces play in classical ($S^0$-)
homotopy theory.
%Given a prime $p$, we will call both $CW_{B\Z/p}$ and $P_{B\Z/p}$
%primary homotopy functors  if the prime $p$ is understood, or simply $B\Z/p$-homotopy functors.

Let $p$ be a prime. If $X$ is a space and we are interested in
describing the $p$-primary part of $X$ through its $A$-homotopy
theory for some space $A$, there are some choices of $A$ that
become apparent. Probably the easiest one are the Moore spaces
$M(\Z /p^n,1)$ and their suspensions; this task was undertaken in
the nineties by Rodriguez-Scherer in the case of cellularization
\cite{MR2003a:55024} and Bousfield \cite{B1}, who did not only described
${P}_{M(\Z /p^m,n)}X$ for a wide number of spaces, but
remarked the close relationship between these functors and the
$v_n$-periodic homotopy theory.

In this paper we deal with the case $A= B\Z /p$. After Miller's
solution of Sullivan conjecture \cite{Miller}, and subsequent work of
Lannes, Dwyer-Zabrodsky and others, a number of new powerful
tools were available to researchers attempting to study the
mapping space $\map(B\Z /p,X)$, and overwhelming success
was reached, particularly for nilpotent spaces $X$. In the framework we
are interested in, we should emphasize the work of Neisendorfer \cite{MR96a:55019}, where
the author proves that the functor ${P}_{ B\Z /p}$ can
often recover the $p$-primary homotopy of $X$ from that of its
$n$-connected cover, or \cite{CCS2}, about the $ B\Z /p$-homotopy of
$H$-spaces.

The first motivation for our work came from two different sources:  the study undertaken by W. Dwyer in \cite{MR97i:55028} concerning $B\Z/p$-nullification of classifying spaces of compact Lie groups whose group of components is a $p$-group, and its relationship with the homology decompositions of $BG$; and Question 11 in Farjoun's book (\cite[page 175]{Dror-Farjoun95}), where he asked about the cellularity of the $p$-completion of $BG$. This seemed a natural extension of the problems considered by the second author concerning $B\Z/p$-homotopy of finite groups (see \cite{Ramon}, \cite{FS07}, and \cite{FFcw}), so it was natural to investigate this structure with similar methods.

%We show that under certain hypothesis, the homology localization functor $L_{\Z[\frac{1}{p}]}$ is related to the %nullification funtor $P_{B\Z/p}$ by means of a fibration.

We show that under certain hypothesis, we are able to characterize the effect of the nullification functor $P_{B\Z/p}$ by means of a fibration.

\medskip
%%%%%%%%%%%%%%%%%%%%%%%%%%%%%%%%%%%%%%%%%%%%%%%%%%%%%%%%%%%
{\bf Theorem~\ref{generalmainthmnull}.}
%%%%%%%%%%%%%%%%%%%%%%%%%%%%%%%%%%%%%%%%%%%%%%%%%%%%%%%%%%%
\noindent {\it Let $X$ be a connected space with finite fundamental group and such that $(P_{B\Z/p}(X\langle 1 \rangle))^\wedge_p\simeq *$. Then there is a fibration
$$L_{\Z[\frac{1}{p}]}(X_p)\rightarrow P_{B\Z/p}(X)\rightarrow B(\pi_1(X)/T_p(\pi_1(X)))$$ where $X_p$ is the covering of $X$ whose fundamental group is $T_p(\pi_1(X))$, and $L_{\Z[\frac{1}{p}]}(X_p)$ denotes Bousfield homological localization of $X_p$ with respect to  $H^*(-;\Z[\frac{1}{p}])$.}
\medskip

In particular, one can compute the homotopy groups of $P_{B\Z/p}(X)$ in terms of those of $X$ if $X$ is good enough. This result is quite general, and in fact describes in a single statement a phenomenon which was previously known for finite groups, $p$-compact groups and some compact Lie groups, but not for $p$-local compact or Ka\v{c}-Moody groups (Corollary \ref{KacMoodygroups}); so, it can be read then as a common property of a big family of homotopy meaningful spaces.  Moreover, finite loop spaces also satisfy this property (Corollary \ref{finiteloopspaces}).

Another source of examples is the theory of infinite loop spaces. McGibbon \cite{McGibbon} shows that infinite loop spaces satisfy the hypothesis of Theorem \ref{generalmainthmnull} (see Corollary~\ref{infiniteloopspaces}).

%\medskip
%%%%%%%%%%%%%%%%%%%%%%%%%%%%%%%%%%%%%%%%%%%%%%%%%%%%%%%%%%%%
%{\bf Corollary~\ref{infiniteloopspaces}.}
%%%%%%%%%%%%%%%%%%%%%%%%%%%%%%%%%%%%%%%%%%%%%%%%%%%%%%%%%%%%
%\noindent {\it Let $X$ be a connected infinite loop space  with finite fundamental group. Then there is a fibration
%$$L_{\Z[\frac{1}{p}]}(X_p)\rightarrow P_{B\Z/p}(X)\rightarrow B(\pi_1(X)/T_p(\pi_1(X)))$$ where $X_p$ is the covering of $X$ whose fundamental group is $T_p(\pi_1(X))$.}

We also obtain some interesting consequences of these results, including a detailed analysis of the relationship between the $B\Z/p$-nullification and $\Z[1/p]$-localization of these spaces -which is very much in the spirit of \cite[Section 6]{MR97i:55028}- and the commutation of nullity functors on them, a situation that was discussed in \cite{RS} in a general framework.

The second part of the paper deals with the effect of the cellularization functor $CW_{B\Z/p}$ on classifying spaces of compact Lie groups. We show a Serre-type dichotomy theorem.

\medskip
%%%%%%%%%%%%%%%%%%%%%%%%%%%%%%%%%%%%%%%%%%%%%%%%%%%%%%%%%%%
{\bf Theorem~\ref{thmdichotomyBG}.}
%%%%%%%%%%%%%%%%%%%%%%%%%%%%%%%%%%%%%%%%%%%%%%%%%%%%%%%%%%%
\noindent {\it Let $G$ be a compact connected Lie group. If there exists a non $p$-cohomologically central element of order $p$, then  the $B\Z /p$-cellullarization of $BG$ has infinitely many nonzero homotopy groups. Otherwise, it has the homotopy type of a $K(V,1)$, where $V$ is a finite elementary abelian $p$-group.}
\medskip

This statement is in fact a consequence of a more general statement which extends Proposition 2.3 in \cite{FS07}.

\medskip
%%%%%%%%%%%%%%%%%%%%%%%%%%%%%%%%%%%%%%%%%%%%%%%%%%%%%%%%%%%
{\bf Theorem~\ref{thmdichotomy}.}
%%%%%%%%%%%%%%%%%%%%%%%%%%%%%%%%%%%%%%%%%%%%%%%%%%%%%%%%%%%
\noindent {\it Let $X$ be a connected nilpotent $\Sigma^n B\Z/p$-null space for some $n\geq 0$. Then the $B\Z /p$-cellullarization of $X$ has the homotopy type of a Postnikov piece with homotopy groups are concentrated in degrees $1$ to $n$ , or else it has infinitely many nonzero homotopy groups. Moreover, if $X$ is $1$-connected of finite type, then the fundamental group $\pi_1(CW_{B\Z/p}(X))$ is a finite elementary abelian $p$-group.}
\medskip

This result opens the way to describe with precision (up to $p$-completion) the $B\Z/p$-cellularization of $BG$ for an ample class of Lie groups which includes p-toral groups and their discrete approximations, the 3-sphere, extensions of elementary abelian groups by groups of order prime to $p$ -which generalize \cite[Corollary 3.3]{FS07}-, or $BSO(3)$. In particular, we find examples of both cases of the dichotomy statement. It is interesting to remark here that we use intensively the fact that $CW_A$ preserves nilpotent spaces (Lemma 2.5), a fact that was conjectured in \cite{Dror-Farjoun95}, but which to our knowledge has not been so far exploited in the literature.

{\bf Notation:} Let $R$ be a commutative ring, $R_\infty(X)$ denotes Bousfield-Kan $p$-completion of a space $X$ (\cite{BK}). When $R=\Z/p$ for a prime $p$, $R_\infty(X)$ will be used instead of $X^\wedge_p$. Moreover, $L_R(X)$  denotes the $HR$-localization of Bousfield (\cite{MR0380779}). All spaces are assumed to have the homotopy type of a $CW$-complex.

\medskip

%%%%%%%%%%%%%%%%%%%%%%%%%%%%%%%
\section{The cellularization and nullification functors}
%%%%%%%%%%%%%%%%%%%%%%%%%%%%%%%

Let $A$ be a connected space. In this section we will define the functors $CW_A$ and $P_A$, which are the main tools we use to describe the $p$-primary structure of the spaces of interest in our work. Only some particular features of these functors, that will be crucial in our
further developments, will be described while their relationship with Bousfield-Kan completion will be studied in the next section. A thorough account to these constructions can be found in \cite{Dror-Farjoun95}.

\begin{defo}
Let $A$ and $X$ be spaces. Then $X$ is called $A$-\emph{null} if the inclusion of constant maps $X\hookrightarrow \map(A,X)$ is a weak equivalence.
\end{defo}

This is equivalent  to the condition that $\map_*(A,X)$ is weakly contractible when $X$ is connected. Dror-Farjoun defines a coaugmented and idempotent functor ${P}_A:\mathbf{Spaces}\ra \mathbf{Spaces}$ where ${P}_AX$ is $A$-null for every $X$, and such that the coaugmentation $X\ra {P}_AX$ induces a weak equivalence $\map({P}_AX,Y)\ra \map(X,Y)$ for every $A$-null space $Y$. The corresponding definitions in the pointed context are completely analogous. Note that in the language of homotopy localization, $P_A$ is the localization with regard to the constant map $A\ra {*}$, and the notation comes from Postnikov sections, which are in fact $S^n$-nullifications. Moreover, a space $X$ such that $P_AX\simeq *$ is called $A$-\emph{acyclic}.

Now we consider the cellular construction, which is somewhat dual of the previous construction, although not completely (see Theorem \ref{Chacho} below).

\begin{defo}
Given pointed spaces $A$ and $X$, $X$ is said $A$-\emph{cellular} if it can be built from $A$ by means of pointed homotopy colimits, possibly iterated. Moreover, a map $X\ra Y$ is said to be an $A$-\emph{equivalence} if it induces a weak equivalence $\map_*(A,X)\ra \map_*(A,Y)$.
\end{defo}

The $A$-cellularization (or $A$-cellular approximation) is a canonical way of turning every space into an $A$-cellular space from the point of view of $A$-equivalences, which generalizes the classic process of cellular approximation. There exists an augmented endofunctor ${CW}_A$ of the category of pointed spaces, such that for every space $X$ the augmentation ${CW}_AX\ra X$ is an $A$-equivalence, and in initial among all maps $Y\ra X$ which induce $A$-equivalence. Unlike ${P}_A$, this functor only makes sense in the pointed  context (\cite[7.4]{Chacholski96}), and can be characterized in several ways \cite[2.E.8]{Dror-Farjoun95}.

The remaining of the section is devoted to describe some properties of these functors that we will frequently use later. We begin with a theorem of W. Chach\'olski that can be considered the most powerful tool to compute cellularization of spaces in an explicit way. The proof can be found in \cite[20.3]{Chacholski96}.

\begin{teo}
\label{Chacho}
Let $A$ and $X$ be pointed spaces, and let $C$ be the homotopy cofibre of the evaluation $\bigvee_{[A,X]_*}A\ra X$, where the wedge is taken
over all the homotopy classes of maps $A\ra X$. Then $CW_AX$ has the homotopy type of the fibre of the map $X\ra P_{\Sigma A}C$.
\end{teo}

Next we will describe two preservation properties, that will be used extensively as we will frequently focus our interest in simply
connected spaces and, more generally, nilpotent spaces.

\begin{lemma}\cite[2.9]{B2}\label{1-connectednullification}
If $X$ is $1$-connected then $P_{A}(X)$ is also $1$-connected.
\end{lemma}

In particular, note that, according to a famous result of Neisendorfer \cite[Thm 0.1]{MR96a:55019}, there is no analogous result for higher degrees of connectivity.

The second preservation property concerns to cellularization and it answers question $7$ stated by Dror-Farjoun in his book \cite[p.175]{Dror-Farjoun95}. It is remarkable that the analogous problem in the category of groups was solved in \cite{FGS07}.

\begin{lemma}\label{nilpotent-cellularization}
If $X$ is a nilpotent space then $CW_A(X)$ is also nilpotent.
\end{lemma}

\begin{proof}
Apply \cite[V.5.2]{BK} to the fibration $CW_A(X)\rightarrow X\rightarrow P_{\Sigma A}C$ in Theorem \ref{Chacho}.
\end{proof}

From the definitions, one can check that if $X$ is $A$-null then $CW_A(X)\simeq *$ since $*\hookrightarrow X$ is an $A$-equivalence. In general, the $A$-cellularization functor also preserves $\Sigma^n A$-nullity for $n\geq 1$.

\begin{lemma}\label{suspension-null}
Let $X$ be a space which is $\Sigma^n A$-null for some $n\ge 1$ then $CW_A(X)$ is also $\Sigma^n A$-null.
\end{lemma}

\begin{proof}
Again from Theorem \ref{Chacho} we have a fibre sequence $CW_A(X)\rightarrow X \rightarrow P_{\Sigma A}(X)$. Since the base space is $\Sigma A$-null, it is also $\Sigma^n A$-null for any $n\geq 1$. The result follows since the nullification functor $P_{\Sigma A}$ preserves then the fibration \cite[3.D.3]{Dror-Farjoun95}.
\end{proof}

If we specialize now to $A= B\Z /p$, which is the case of interest in this paper, and we turn our attention to Eilenberg-MacLane spaces, it is interesting to observe that given an arbitrary group $G$, the $B \Z /p$-nullity properties of  $K(G,n)$ for small values of $n$ imply
the $B\Z /p$-nullity for $every$ value of $n$, as well as some group-theoretic features of $G$.

\begin{lemma}\label{nullK(G,n)}
Let $G$ be an abelian discrete group.  $K(G,2)$ is $B\Z/p$-null if and only if $p$ is invertible in $G$ and $K(G,n)$ is $B\Z/p$-null for all $n\geq 1$.
\end{lemma}

\begin{proof}
We only need to show that if $K(G,2)$ is $B\Z/p$-null then $p$ is invertible in $G$ and $K(G,n)$ is $B\Z/p$-null for all $n\geq 1$. Since $K(G,1)\simeq \Omega K(G,2)$ is $B\Z/p$-null, $Hom(\Z/p,G)=[B\Z/p,BG]_*=0$. Therefore $G$ has no elements of order $p$. Then, multiplication by $p$ gives rise to a short exact sequence $0\rightarrow G\stackrel{p}{\rightarrow} G\rightarrow G/pG\rightarrow 0$.  Now consider the induced fibration $K(G,1)\rightarrow K(G/pG,1)\rightarrow K(G,2)$. Since both $K(G,1)$ and $K(G,2)$ are $B\Z/p$-null, by \cite[3.D.3]{Dror-Farjoun95},  we see that $B(G/pG)$ is also $B\Z/p$-null. Therefore $G/pG$ has no elements of order $p$, so it must be trivial. That is $G\stackrel{p}{\rightarrow}G$ is an isomorphism and $p$ is invertible in $G$.

A standard transfer argument (see e.g. \cite[Prop III.10.1]{MR672956}) shows that $\tilde{H}^*(B\Z/p;G)$ is trivial. In particular, $\map_*(B\Z/p,K(G,n))$ is weakly contractible for all $n\geq 1$.
\end{proof}

We finish this preliminary section by describing a context in which we can obtain information about the homology and homotopy groups of the cellularization.

\begin{lemma}\label{acyclic-cellularization}
If $R$ is a ring of coefficients and $\tilde{H}^*(A;R)=0$, then $\tilde{H}^*(CW_A(X);R)=0$. If $X$ is nilpotent and $R\subset \Q$ then $\pi_i(CW_A(X))\otimes R=0$ for $i>0$.
\end{lemma}

\begin{proof}
Under the hypothesis of the theorem, $K(R,n)$ is $A$-local for $n>0$, then the space $\map_*(CW_A(X), K(R,n))$ is weakly contractible. By Lemma \ref{nilpotent-cellularization}, we can apply \cite[V.3.1]{BK}.
\end{proof}

%%%%%%%%%%%%%%%%%%%%%%%%%%
\section{$B\Z/p$-homotopy and  $p$-completion}
%%%%%%%%%%%%%%%%%%%%%%%%%%

 We devote this section to the description of the behaviour of the functors $CW_A$ and $P_A$ with respect to the $p$-completion functor of Bousfield and Kan. In particular, if $\eta\colon X\rightarrow X^{\wedge}_p$ is the $p$-completion, we want to characterize when the maps $CW_A(\eta)$ and $P_A(\eta)$ are mod $p$ equivalences. This will be fundamental in our approach to the $B\Z /p$-nullification and $B\Z /p$-cellularization of classifying spaces, which will be undertaken in the last two sections and is the main goal of our note.

A first approximation to these kind of questions appears in the work of  Miller in the solution of the Sullivan Conjecture, which implies immediately a statement about $B\Z /p$-nullity.

\begin{teo}\cite[Thm 1.5]{Miller} \label{bastaenp-Miller}
Let $W$ be a connected space with $\tilde{H}^*(W;\Z[\frac{1}{p}])=0$ and let $X$ be a nilpotent space. Then $\eta\colon X\rightarrow X^{\wedge}_p$ is a $W$-equivalence.
\end{teo}

\begin{cor}
\label{bastaenp}
If $X$ is a nilpotent space, the $p$-completion $\eta \colon X\rightarrow X^{\wedge}_p$ is a
$\emph{B}\Z /p$-equivalence.
\end{cor}

%\begin{remark}
%Theorem \ref{bastaenp-Miller} and Corollary \ref{bastaenp} also work for spaces which are virtually nilpotent, in particular for spaces with finite fundamental group (see \cite[p.188]{Schwartz94}).
%\end{remark}

%\begin{proof}
%If $X$ is $1$-connected, then it is $p$-good and $X^{\wedge}_p$ is also simply connected (\cite[VII.6.6]{BK})
%Using a Sullivan arithmetic square argument (see \cite[Thm 1.5]{Miller}), Miller proved that the $p$-completion $\eta\colon X\rightarrow X^{\wedge}_p$ induces a homotopy equivalence between pointed mapping spaces
%$\map_*( B\mathbb{Z}/p,X) \simeq \map_*( B\mathbb{Z}/p,X^{\wedge}_p)$, and therefore
%$CW_{ B\mathbb{Z}/p}
%X\simeqCW_{ B\mathbb{Z}/p} X^{\wedge}_p$.
%\end{proof}

Observe that if $X$ is $1$-connected, we can $p$-complete our target space, if necessary, before computing $CW_{B\Z /p}X$. This statement, and the fact that the $B\Z /p$-cellularization is constructed using copies of $B\Z /p$ as pieces, may lead to think that $CW_{B\Z /p}X$ is always a $p$-complete space. Next lemma shows that this is true in certain cases but, as we will see in Example \ref{no1connected}, not always.

\begin{lemma}\label{p-complete-cellularization}
If $X$ is a nilpotent space, then $CW_{B\Z/p}(X)$ is $p$-complete if and only if $\tilde{H}_*(CW_{B\Z/p}(X)^{\wedge}_p;\Q)=0$.
\end{lemma}

\begin{proof}
 Since $B\Z /p$ is both $\mathbb{Q}$-acyclic and $\mathbb{F}_q$-acyclic for $q\neq p$, $CW_{B\Z/p}(X)$ is so (\cite[D.2.5]{Dror-Farjoun95} or Lemma \ref{acyclic-cellularization}), and then the rationalization and $q$-completions of $CW_{B\Z/p}(X)$ are trivial. By Lemma \ref{nilpotent-cellularization}, $CW_{B\Z/p}(X)$ is also a nilpotent space, so it admits a Sullivan arithmetic square decomposition. The result follows.
\end{proof}

%Now we will see by means of counterexamples that both simply-connectedness of $X$ and $\tilde{H}(X^{\wedge}_p;\Q)=0$ are necessary conditions for obtaining $p$-complete $B\Z/p$-cellularizations.

\begin{ejem}
\label{no1connected}
Consider the space $X=K(\Z/p^\infty, 2)$. $X$ is $B\Z/p$-cellular by \cite[Lemma 3.3]{CCS2}, but it is not $p$-complete since $X^\wedge_p\simeq K(\Z^\wedge_p, 3)$ and $\tilde{H}^*(X^\wedge_p;\Q)\neq 0$. In fact, the $p$-completion $\eta\colon K(\Z/p^\infty, 2)\rightarrow K(\Z^\wedge_p, 3)$ induces a $B\Z/p$-cellular equivalence, then $CW_{B\Z/p}(X^\wedge_p)\simeq X$. On the other hand, taking for example $p=2$ and $X=B\Sigma_3$, the
classifying space of the symmetric group in three letters, it is not nilpotent, and the cellularization is not complete. See \cite[Example 2.6]{FS07} for details.
\end{ejem}

%For the second part, consider again Chach\'olski's cofibration
%$\bigvee B\mathbb{Z}/p\stackrel{f}{\longrightarrow}
%X\longrightarrow C_f$. As $\bigvee B\mathbb{Z}/p$ has the
%rational homology of a point, the Mayer-Vietoris sequence of the
%cofibration proves that $f$ is a rational homology equivalence,
%and hence induces a homotopy equivalence after rationalization
%(note that the spaces involved are simply-connected). As
%$\Sigma B\mathbb{Z}/p\longrightarrow {*}$ is again a
%rational homology equivalence, we have that $X_{\mathbb{Q}}\simeq
%(P_{\Sigma  B\mathbb{Z}/p}C_f)_{\mathbb{Q}}$.

%Now, consider the fibre sequence:
%$$CW_{ B\mathbb{Z}/p} X\longrightarrow
%X\longrightarrowP_{\Sigma B\mathbb{Z}/p}C_f.$$ As
%the base is simply connected, the fibration is preserved by
%rationalization, and the statement is now an easy consequence of
%Zeeman's comparison theorem.

%\begin{remark}
%\label{no1connected}
%The result is not true in general if $X$ is not simply connected,
%as you can see by taking for example $p=2$ and $X= B\Sigma_3$, the
%classifying space of the symmetric group in three letters. See \cite[Ex 2.6]{FS07} for details.
%\end{remark}

We proceed now to a systematic study of the induced
map $CW_{B\Z/p}(\eta)\colon CW_{B\Z/p}(X)\rightarrow CW_{B\Z/p}(X^{\wedge}_p)$. We want to show under which conditions it becomes a mod $p$ equivalence. The first step is a reduction concerning the fundamental group, for which we need the following definition.

\begin{defo}
We say that an element $x\in \pi_1(X)$ lifts to $X$ if there exists a homotopy lift
$$
\xymatrix{ & X \ar[d] \\ \emph{B}(\langle x \rangle) \ar[ur] \ar^{i_{\langle x \rangle}}[r] & \emph{B}(\pi_1(X)).}
$$

\end{defo}

\begin{prop}\label{liftin}
Let X be a connected space. There is a fibration $$CW_{B\Z/p}(X)\rightarrow X \rightarrow Z$$ with $\pi_1(Z)\cong \pi_1(X)/S$, where $S$ is the normal subgroup generated by the elements of order $p$ which lift to $X$.
%Moreover, if $E$ is the pullback of $BS \rightarrow B\pi_1(X) \leftarrow X$, then $\pi_1(E)\cong S$ and $E\rightarrow X$ is a $B\Z/p$-cellular equivalence.
\end{prop}

\begin{proof}
The fibration in the proposition is the one constructed by Chach\'{o}lski (see Theorem \ref{Chacho}) where $Z=P_{\Sigma B\Z/p}(C_X)$.
The subgroup $S$ is constructed in \cite[Prop. 2.1]{CCS2} in a way that $E\rightarrow X$ is a $B\Z/p$-cellular equivalence, where $E$ is the homotopy pullback
$$
\xymatrix{
E \ar[r]^{} \ar[d]^{} &
X  \ar[d]^{ }  \\
BS\ar[r]^{Bi} & B(\pi_1(X)).
}$$
By construction $\pi_1(E)\cong S$ is generated by elements of order $p$ which lift to $E$. Then the Chach\'{o}lski's cofibre $C_E$ (see Theorem \ref{Chacho}) is $1$-connected and $P_{\Sigma B\Z/p}(C_E)$ is too by Lemma \ref{1-connectednullification}. Since $E\rightarrow X$ is a $B\Z/p$-equivalence, from the following diagram of fibrations
$$
\xymatrix{
CW_{B\Z/p}E \ar[r]^{\simeq} \ar[d]^{} &
CW_{B\Z/p}X  \ar[d]^{ }  \\
E \ar[r]^{} \ar[d]^{} &
X  \ar[d]^{ }  \\
P_{\Sigma B\Z/p}(C_E)\ar[r]^{ } & P_{\Sigma B\Z/p}(C).
}$$
where $C$ is Chach\'{o}lski's cofibre for $X$, we see that the fundamental group of $P_{\Sigma B\Z/p}(C)$ is $\pi_1(X)/S$.
\end{proof}

\begin{cor}\label{unpointed}
Let $X$ be a connected space such that $\pi_1(X)$ is generated by elements of order $p$ which lift to $X$. There is a bijection $[B\Z/p,CW_{B\Z/p}(X)]\cong[B\Z/p,X]$ between unpointed homotopy classes of maps.
\end{cor}

\begin{proof}
Since $CW_{B\Z/p}(X)\rightarrow X$ is a $B\Z/p$-homotopy equivalence, there is a bijection $[B\Z/p,CW_{B\Z/p}(X)]_*\cong[B\Z/p,X]_*$ between pointed homotopy classes of maps. The following diagram
$$
\xymatrix{
 [B\Z/p,CW_{B\Z/p}(X)]_*\ar[r]^{} \ar[d]^{} &
[B\Z/p,X]_*  \ar[d]^{ }  \\
[B\Z/p,CW_{B\Z/p}(X)]\ar[r] & [B\Z/p,X]
}$$
shows that the quotient map is also a bijection since the induced morphism on fundamental groups $\pi_1(CW_{B\Z/p}(X))\rightarrow \pi_1(X)$ is an epimorphism by Proposition \ref{liftin}.
\end{proof}

We can get information about the fundamental group of the cellularization since being $B\Z/p$-cellular imposes some restrictions on the fundamental group of the space.

\begin{lemma}\label{fundamentalgroupCWlifts}\label{fundamentalgroupofCW}
 If $X$ is a $B\Z/p$-cellular space, its fundamental group is generated by elements of order $p$ which lift to $X$. Moreover, if  $X$ is a finite type $1$-connected  space,then $\pi_1(CW_{B\Z/p}(X))$ is a finitely generated abelian generated by elements of order $p$ which lift to $CW_{B\Z/p}(X)$.
\end{lemma}

\begin{proof}
Let $S$ be the normal subgroup of $\pi_1(X)$ generated by elements of order $p$ which lift to $X$. Consider the pullback diagram $$\xymatrix{
E \ar[d] \ar[r]  & X \ar[d]  \\
BS \ar[r]^{Bi}  & B\pi_1(X).
}$$

By \cite[Prop 2.1]{CCS2}, the map $E\rightarrow X$ is a $B\Z/p$-cellular equivalence. Since $X$ is $B\Z/p$-cellular, there exists a map $f\colon E\rightarrow CW_{B\Z/p}(E)$ such that $i\circ f\simeq id$ where $i\colon CW_{B\Z/p}(E)\rightarrow E$. In fact, this implies that $p\colon E\rightarrow P_{\Sigma B\Z/p}(C_E)$ is nullhomotopic, $p\simeq p \circ i\circ f \simeq *\circ f \simeq *$, therefore $CW_{B\Z/p}(E)\simeq E\times P_{B\Z/p}(\Omega C_E)$. But this implies that $E$ is $B\Z/p$-cellular since $CW_{B\Z/p}(E)$ is $B\Z/p$-acyclic, and then $E\simeq X$. In particular, $\pi_1(X)=S$.

To prove the second statement, it remains to prove that $\pi_1(CW_{B\Z/p}(X))$ is a finitely generated abelian group. Since $X$ is $1$-connected, then the Chach\'{o}lski's cofibre $C_X$ is $1$-connected and $P_{\Sigma B\Z/p}(C_X)$ is too by Lemma \ref{1-connectednullification}. Then, we see that $\pi_2(P_{\Sigma B\Z/p}(C_X))\cong H_2(P_{\Sigma B\Z/p}(C_X);\Z)$ is a quotient of $H_2(C_X;\Z)$, which in turn is a quotient of the finitely generated group $H_2(X;\Z)$.
\end{proof}

%\begin{lemma}\label{fundamentalgroupofCW}
%Let $X$ be a finite type $1$-connected  space, then $\pi_1(CW_{B\Z/p}(X))$ is a finitely generated abelian generated by elements of order $p$ which lift to $CW_{B\Z/p}(X)$.
%\end{lemma}

%\begin{proof}
%By the previous Lemma \ref{fundamentalgroupCWlifts}, it remains to prove that $\pi_1(CW_{B\Z/p}(X))$ is a finitely generated abelian group. Since $X$ is $1$-connected, then the Chach\'{o}lski's cofibre $C_X$ is $1$-connected and $P_{\Sigma B\Z/p}(C_X)$ is too by Lemma \ref{1-connectednullification}. Then $\pi_2(P_{\Sigma B\Z/p}(C_X))\cong H_2(P_{\Sigma B\Z/p}(C_X);\Z)$ is a quotient of $H_2(C_X;\Z)$, which in turn is a quotient of the finitely generated group $H_2(X;\Z)$.
%\end{proof}

Next we need a technical lemma which describes the somewhat intrincate relationship between completion and nullification and that is a key result to understand under which conditions $P_{A}(\eta)\colon P_{A}(X)\rightarrow P_{A}(X^{\wedge}_p)$ is a mod $p$ equivalence (Corollary \ref{null+modp}).

\begin{lemma}\label{null+comp}
Let $A$ be a connected space, and let $X$ such that $P_A(X^{\wedge}_p)$ and $P_A(X)$ are $p$-good spaces. Assume that $P_A(X)^{\wedge}_p$ and  $P_A(X^{\wedge}_p)^{\wedge}_p$ are $A$-null spaces. Then the $p$-completion map $\eta_X\colon X\rightarrow X^{\wedge}_p$ induces a mod $p$ equivalence $P_A(\eta)\colon P_A(X)\rightarrow P_A(X^{\wedge}_p)$.
\end{lemma}

\begin{proof}
Let $\epsilon \colon P_A(X^{\wedge}_p)\rightarrow (P_A(X))^{\wedge}_p$ be the unique map up to homotopy such that the right square of the following diagram commutes:
$$
\xymatrix{
X \ar[r]^{\eta_X} \ar[d]^{\iota_X} &
X^{\wedge}_p \ar[r]^{id} \ar[d]^{\iota_{X^{\wedge}_p}} &
X^{\wedge}_p \ar[d]^{(\iota_X)^{\wedge}_p} \\
P_A(X)\ar[r]^{P_A(\eta_X)} & P_A(X^{\wedge}_p) \ar[r]^{\epsilon} & P_A(X)^{\wedge}_p.
}$$

Note that $\epsilon$ exists because $P_A(X)^{\wedge}_p$ is $A$-null by hypothesis. The left square commutes by naturality, so $(\iota_X)^{\wedge}_p\circ \eta_X \simeq \epsilon \circ P_A(\eta_X) \circ \iota_X$. But also, $(\iota_X)^{\wedge}_p\circ \eta_X \simeq  \eta_{P_A(X)} \circ \iota_X$ by naturality of the completion. Because of the universal property of the nullification functor, $\epsilon \circ P_A(\eta_X)\simeq \eta_{P_A(X)}$. Since $P_A(X)$ is $p$-good, $ \eta_{P_A(X)}^*$ is an isomorphism in mod $p$ cohomology. In particular, $\epsilon^*$ is  a monomorphism and $P_A({\eta}_X)^*$ is an epimorphism.

Now consider the following commutative diagram:
$$
\xymatrix{
X^{\wedge}_p \ar[r]^{id} \ar[d]^{\iota_{X^{\wedge}_p}} &
X^{\wedge}_p \ar[r]^{(\eta_X)^{\wedge}_p} \ar[d]^{(\iota_{X})^{\wedge}_p} &
X^{\wedge}_p \ar[d]^{(\iota_{X^{\wedge}_p})^{\wedge}_p} \\
P_A(X^{\wedge}_p)\ar[r]^{\epsilon} & P_A(X)^{\wedge}_p \ar[r]^{P_A(\eta_X)^{\wedge}_p} & P_A(X^{\wedge}_p)^{\wedge}_p.
}$$
That is $(\iota_{X^{\wedge}_p})^{\wedge}_p \circ (\eta_X)^{\wedge}_p\simeq P_A(\eta)^{\wedge}_p \circ \epsilon \circ \iota_{X^{\wedge}_p}$. But we also have $(\iota_{X^{\wedge}_p})^{\wedge}_p \circ (\eta_X)^{\wedge}_p \simeq (\eta_{P_A(X^{\wedge}_p)})\circ \iota_{X^{\wedge}_p}$. By hypothesis $P_A(X^{\wedge}_p)^{\wedge}_p$ is $A$-null,  then the universal property of the nullification functor  implies that  $P_A({\eta}_X)^{\wedge}_p \circ \epsilon \simeq \eta_{P_A(X^{\wedge}_p)}$. Since $P_A(X^{\wedge}_p)$ is $p$-good,  $(\eta_{P_A(X^{\wedge}_p)})^*$ is an isomorphism and hence $(P_A(\eta_X)^{\wedge}_p)^*$ is a monomorphism. Therefore $P_A(\eta_X)^*$ is so, and we are done.
\end{proof}

\begin{remark}
If $X$ has finite fundamental group, then both  $P_A(X^{\wedge}_p)$ and $P_A(X)$ are $p$-good spaces since they also have finite fundamental groups.
\end{remark}

\begin{cor}\label{null+modp}
If $X$ is a $1$-connected space and $A$ is such that $\tilde{H}_*(A;\Z[\frac{1}{p}])=0$ then $P_{A}(\eta)\colon P_{A}(X)\rightarrow P_{A}(X^{\wedge}_p)$ is a mod $p$ equivalence.
\end{cor}

\begin{proof}
If $X$ is $1$-connected then $X^{\wedge}_p$ is also $1$-connected and both spaces are $p$-good. Moreover the $B\Z/p$-nullification of a $1$-connected space is also $1$-connected.  Miller's theorem \cite[Thm 1.5]{Miller} implies that the spaces $P_{A}(X)^{\wedge}_p$ and $P_{A}(X^{\wedge}_p)^{\wedge}_p$ are $A$-null. The hypothesis of Lemma \ref{null+comp} are then satisfied.
\end{proof}

We can also describe a general situation in which the nullification of a mod $p$ equivalence is so.

\begin{cor}\label{null+modp+map}
Let $A$ be a space such that $\tilde{H}_*(A;\Z[\frac{1}{p}])=0$. If $f\colon X\rightarrow Y$ is a mod $p$ equivalence between $1$-connected spaces then $P_{A}(f)\colon P_{A}(X)\rightarrow P_{A}(Y)$  is a mod $p$ equivalence.
\end{cor}

\begin{proof}
If $f$ is a mod $p$ equivalence, then $f^{\wedge}_p$ is an equivalence. Then the following diagram commutes
$$\xymatrix{
 P_{A}(X)\ar[d]_{P_{A}(\eta_X)} \ar[r]^{P_{A}(f)} &  P_{A}(Y) \ar[d]^{P_{A}(\eta_Y)} \\
 P_{A}(X^{\wedge}_p) \ar[r]_{P_{A}(f^{\wedge}_p)} &  P_{A}(Y^{\wedge}_p)
}$$
By Corollary \ref{null+modp}, the two vertical arrows are mod $p$ equivalences and the bottom horizontal map is an equivalence. Then $P_{A}(f)$ is a mod $p$ equivalence.
\end{proof}

\begin{remark}\label{remark+null+map}
Note that Corollaries \ref{null+modp} and \ref{null+modp+map} hold when $A=B\Z /p$. In fact, in Corollary \ref{null+modp+map}, one can relax the assumptions on $1$-connectivity by checking that both spaces $X$ and $Y$ satisfy the assumptions of Lemma \ref{null+comp}.
\end{remark}

Now we follow the parallelism giving a condition for the analogous equivalence between cellularizations to hold. According to Proposition \ref{liftin}, the hypothesis of lifting elements in the fundamental group is not a real restriction.

\begin{prop} \label{cw+comp}
Let $X$ be a space whose fundamental group  $\pi_1(X)$ is finite and generated by elements of order $p$ which lift to $X$. Assume that there is a bijection $[B\Z/p,X]=[B\Z/p,X^{\wedge}_p]$,  then the map induced by the $p$-completion $$CW_{B\Z/p}(\eta)\colon CW_{B\Z/p}(X)\rightarrow CW_{B\Z/p}(X^{\wedge}_p)$$ is a mod $p$ equivalence.
\end{prop}

\begin{proof} Since $\pi_1(X)$ is finite, $X$ is $p$-good \cite[VII.5.1]{BK}. There is an epimorphism $\pi_1(X)\rightarrow \pi_1(X^{\wedge}_p)$ and, by assumption, $[B\Z/p,X]\cong [B\Z/p,X^{\wedge}_p]$. In order to compute the cellularization, we first analyze Chach\'{o}lski's cofibres

$$\xymatrix{
\vee B\Z/p \ar[r]^{h_1} \ar[d]^{id} & X \ar[d]^{\eta} \ar[r] & C\ar[d]^{g} \\
\vee B\Z/p \ar[r]^{h_2} & X^{\wedge}_p \ar[r] & D.
}$$

Since $\pi_1(X)$ is finite and generated by elements of order $p$ which lift to $X$, the maps $h_1$ and $h_2$ induce epimorphisms on the fundamental group and then $C$ and $D$ are 1-connected spaces. Moreover $g$ is a mod $p$ equivalence. Now, the cellularization fits in the following diagram of fibrations:
$$\xymatrix{
CW_{B\Z/p}(X) \ar[r] \ar[d]^{CW_{B\Z/p}(\eta)} & X \ar[d]^{\eta} \ar[r] & P_{\Sigma B\Z/p}(C)\ar[d]^{P_{\Sigma B\Z/p}(g)} \\
CW_{B\Z/p}(X^{\wedge}_p) \ar[r] & X^{\wedge}_p \ar[r] & P_{\Sigma B\Z/p}(D),
}$$
 where $P_{\Sigma B\Z/p}(C)$ and $P_{\Sigma B\Z/p}(D)$ are $1$-connected. All the spaces in the previous diagram are $p$-good. Therefore, to show that $CW_{B\Z/p}(\eta)$ is a mod $p$ equivalence, it is enough to prove that $P_{\Sigma B\Z/p}(g)$ is so. This follows from the previous Corollary \ref{null+modp+map} since $g$ is a mod $p$-equivalence.
 \end{proof}

\begin{remark}
We note in Example \ref{no1connected} that $CW_{B\Z/p}(X^\wedge_p)$ does not need to be $p$-complete and a condition for this to be true was stated in Lemma
\ref{p-complete-cellularization}. If $X$ satisfies the hypothesis of Proposition \ref{cw+comp}, we see from the proof that $CW_{B\Z/p}(X^\wedge_p)$ is $p$-complete if $P_{\Sigma B\Z/p}(D)$ is so. This last space is $1$-connected and, using an arithmetic Sullivan square argument, we see that this is the case if $X^\wedge_p\rightarrow P_{\Sigma B\Z/p}(D)^\wedge_p$ is a rational equivalence. Examples of this situation are provided by classifying spaces of finite groups, since $(BG^\wedge_p)_\Q\simeq *$ and  $P_{\Sigma B\Z/p}(D)^\wedge_p\simeq P_{\Sigma B\Z/p}(C)^\wedge_p$ is homotopic to the $p$-completion of the classifying space of a finite group by \cite[Proposition 5.5]{FS07} and \cite[Theorem 4.3]{FFcw}.
\end{remark}

\begin{remark}
 The hypothesis of Proposition \ref{cw+comp} are satisfied if $\pi_1(X)$ is a finite $p$-group generated by elements of order $p$ which lift to $X$ (see \cite[Proof of 3.1]{DZ}).
\end{remark}

\begin{remark}\label{remark-p-toral}
Let $P$ be a $p$-toral group. Then $P$ is  an extension of a finite $p$-group $\pi$ by a torus $(S^1)^n$. Assume that $\pi$ is generated by elements of order $p$ which lift to $BP$. The arguments of \cite[proof of 3.1]{DZ} applied to the fibration $B(S^1)^n\rightarrow BP\rightarrow B\pi$ show that $[B\Z/p,BP]=[B\Z/p,BP^{\wedge}_p]$. Then $BP^\wedge_p$ is a $p$-compact toral group, and $CW_{B\Z/p}(\eta)\colon CW_{B\Z/p}(BP)\rightarrow CW_{B\Z/p}(BP^{\wedge}_p)$ is a mod $p$ equivalence.

By \cite[Proposition 6.9]{DW-annals}, there exists a discrete $p$-toral group $P_\infty$, that is, an extension of a finite $p$-group $\pi$ by a finite sum of Pr\"ufer groups $(\Z/p^\infty)^n$, such that $BP_\infty\rightarrow BP^\wedge_p$ is a mod $p$ equivalence. Moreover, $[B\Z/p,BP_\infty]\cong [B\Z/p,(BP_\infty)^\wedge_p]$ by \cite[Remark 6.12]{DW-annals}, so we should study, up to $p$-completion, the $B\Z/p$-cellullarization of discrete $p$-toral groups. See Example \ref{cw-p-toral}.
% The arguments of \cite[Proof of 3.1]{DZ} applied to the fibration $B(\Z/p^\infty)^n\rightarrow BP\rightarrow B\pi$ show that $[B\Z/p,BP]=[B\Z/p,BP^{\wedge}_p]$, and then Proposition \ref{cw+comp} gives the mod $p$ equivalence $CW_{B\Z/p}(BP)\simeq CW_{B\Z/p}(BP^{\wedge}_p)$. This will be useful in Examples \ref{O(2)} and \ref{cw-p-toral}.
\end{remark}

%%%%%%%%%%%%%%%%%%%%%%%%%%%%%%%%%%
\section{$B\Z/p$-nullification of classifying spaces}
%%%%%%%%%%%%%%%%%%%%%%%%%%%%%%%%%%

In this section, we are concerned with $ B\Z
/p$-nullification. The original motivating example for our study were classifying spaces of compact Lie groups, for which Dwyer computed in \cite{MR97i:55028}
the value of $P_{B\Z /p}BG$ for the case in which $\pi_0(G)$ is a (finite) $p$-group. For this sake, he used an induction principle based on the
centralizer decomposition of $BG$, a method that also
solve the problem when we take a $p$-compact group $X$ instead of
$G$. However, the hypothesis over the fundamental group is
essential and cannot be removed from his proof, so we need to
follow a completely different path to solve the general case. In fact, our new strategy was useful to describe $P_{B\Z /p}X$ for a bigger family of spaces, which in particular need not to be classifying spaces.

%\begin{defo}\label{detectingfamiliy}
% Consider a space $X$ and a family of spaces $Y_i$, $i\in I$. We say that $\{f_i\colon Y_i\rightarrow X\}_{i\in I}$ is a $A$-detecting familiy at $p$ for $X$ if for any  $g\colon X\rightarrow Z$ with $Z$ an $A$-null and $p$-complete space, we have that $g\simeq *$ iff $g\circ f_i\simeq *$ for all $i\in I$. We say that this family is non-trivial if $Y_i\neq X$ for all $i$.
%\end{defo}

%\begin{ejem}\label{ejemLie}
%In \cite[Thm 1.4]{MR97i:55028}, Dwyer shows that if $G$ is a compact Lie group then $\{f_H\colon BH\rightarrow BG\}_{H\leq G}$, where $H$ is a finite $p$-group, is a $\Sigma B\Z/p$-detecting family at $p$ for $BG$.
%\end{ejem}

%\begin{ejem}\label{ejempushout}
%Let $Y$ be the pushout of a diagram of spaces $Y_1\leftarrow Y_0 \rightarrow Y_2$. Then $\{f_i\colon Y_i\rightarrow Y\}_{i=0,1,2}$ is a $A$-detecting family at $p$ for any connected space $A$.
%\end{ejem}

%\begin{ejem}\label{ejemacyclic}
%If $X$ is an $A$-acyclic space then $\{*\rightarrow X\}$ is an $A$-detecting family at $p$ for any $p$ since all morphisms $X\rightarrow Z$ are nullhomotopic if $X$ is $A$-null.
%\end{ejem}

Recall that, if $S$ is a set of primes, the $S$-radical subgroup $T_S(G)$ of a finite group $G$ is the smallest normal subgroup of $G$ which contains all the $S$-torsion. This is the last ingredient we need to state the main result of this section.

\begin{teo}\label{generalmainthmnull}
Let $X$ be a connected space with finite fundamental group and such that $P_{B\Z/p}(X\langle 1 \rangle)\simeq *$. Then there is a fibration
$$L_{\Z[\frac{1}{p}]}(X_p)\rightarrow P_{B\Z/p}(X)\rightarrow B(\pi_1(X)/T_p(\pi_1(X)))$$ where $X_p$ is the covering of $X$ whose fundamental group is $T_p(\pi_1(X))$, and $L_{\Z[\frac{1}{p}]}(X_p)$ denotes the homological localization of $X_p$ in the ring $\Z[\frac{1}{p}]$.
\end{teo}

Theorem \ref{generalmainthmnull} will be a consequence of the following result.

\begin{teo}\label{generalthmnull}
Let $X$ be a connected space with finite fundamental group generated by $p$-torsion elements which lift to $X$ and such that $P_{B\Z/p}(X\langle 1 \rangle)\simeq *$. Then there is an equivalence
$P_{B\Z/p}(X)\rightarrow L_{\Z[\frac{1}{p}]}(X),$ where $L_{\Z[\frac{1}{p}]}(X)$.
\end{teo}

%The next Lemma, which as we have said is an important step in the proof of the previous Corollary,

Now in order to prove Theorem \ref{generalthmnull} we follow the strategy of the second author in \cite{Ramon} when dealing with classifying spaces of finite groups, although now there is rational information that is absent in the finite case. Before, however, we will be deal with some issues concerning to the fundamental group of $X$ which will be crucial in the proof.

\begin{lemma}
\label{radicalperfect}
Let $G$ be a finite group and $S$ a set of primes that divide the
order of $G$. If $G=\emph{T}_SG$, then $G$ is
$S^{-1}$-perfect. In particular, if $X$ is a space with finite fundamental group such that $\pi_1X=T_S(\pi_1(X))$, then $L_{\mathbb{Z}[S^{-1}]}(X)$ is simply-connected.
\end{lemma}

\begin{proof}
%Consider the $\mathbb{Z}[S^{-1}]$-perfect maximal subgroup $\pi_S$
%of $G$. According to REF-BK, the quotient $G/\pi_SG$ has no
%$S$-torsion. Hence, by the definition of radical, $T_SG\unlhd G$
%and the result follows.
The first statement follows from the fact that, since $G$ is generated by
$S$-torsion, $G_{ab}$ is an abelian finite $S$-torsion subgroup,
and then $\mathbb{Z}[S^{-1}]\otimes G^{ab}=0$.

For the second statement, observe that as $G$ is $S^{-1}$-perfect, then $X$
is a $\mathbb{Z}[S^{-1}]$-good space, and
the $\Z[S^{-1}]$-completion of $X$ is $1$-connected, by \cite[VII.3.2]{BK}. But for a connected $\Z[\frac{1}{p}]$-good space $X$, the  $\Z[\frac{1}{p}]$-completion is an $H_*(-;\Z[\frac{1}{p}])$-localization (see \cite[page 205]{BK}).
\end{proof}

In particular, if $X$ is a connected space such that its fundamental group
is finite and equal to its $\mathbb{Z}/p$-radical, then
$L_{\mathbb{Z}[1/p]}X$ is a simply-connected space.

%WEAKENAR HIPOTESIS
\begin{lemma}
\label{univcoefficients}
Let $X$ be a connected space and $p$ a prime. Then the coaugmentation $X\ra L_{\Z [1/p]}X$ is an $\mathbb{F}_q$-equivalence and a $\Q$-equivalence where $q$ is a prime such that $(q,p)=1$. If $L_{\Z [1/p]}X$ is $1$-connected then $L_{\Z [1/p]}X$ is $\F_p$-acyclic.
\end{lemma}

\begin{proof}
By universal coefficient theorem (e.g. see \cite[5.2.15]{Spanier}), the coaugmentation $X\ra L_{\Z [1/p]}X$ is a $G$-equivalence for any $\Z[\frac{1}{p}]$-module $G$.  The last statement follows form \cite[Lemma 6.2]{MR97i:55028}.
%For the second statement, observe that since
%$\pi_1(X)$ is finite, it is $\Q$-perfect. Hence, $X$ and
%$L_{\Z [1/p]}X$ are $\Q$-good (see \cite[VII.3.2]{BK}), and then we have the
%equivalences $X_{\Q}\simeq L_{\Q}X\simeq L_{\Q}L_{\Z
%[1/p]}X\simeq (L_{\Z [1/p]}X)_{\Q}$. So, the coaugmentation
%induces a $\Q$-homology equivalence, and we are done.
\end{proof}

%\begin{lemma}
%\label{nullfamily}
%Let $Z$ be a $p$-complete $B\Z/p$-null space and $X$ be a connected space. Assume that $\{f_i\colon X_i\rightarrow X\}$ is a $\Sigma B\Z/p$-detecting family at $p$ such that  $P_{B\Z/p}(X_i)\simeq *$ for all $i$. Then any map $f\colon X\rightarrow Z$ is nullhomotopic.
%\end{lemma}

%\begin{proof}
%By definition $f$ is nullhomotopic iff for all $i$ $f\circ f_i\colon X_i\rightarrow Z$ is so. But since $Z$ is $B\Z/p$-null, this map factors through $P_{B\Z/p}(X_i)\simeq *$.
%\end{proof}

%We will need a little bit more sophisticated version of this result, which is at the same time a consequence of it:

\begin{lemma}
\label{fundamentalgroups}
Let $Z$ be a $B\Z/p$-null space and $X$ be a connected space such that $\pi_1(X)$ is a finite group generated by $p$-torsion elements which lift to $X$.  Then for any $f\colon X\rightarrow Z$, the composite $X\rightarrow Z \rightarrow B\pi_1(Z)$ is nullhomotopic.
\end{lemma}

\begin{proof}
Let $f\colon X\rightarrow Z$ be any map. We must check that $\pi_1(f)$ is the trivial morphism. It is enough to show that the map between unpointed homotopy classes $[S^1,X]\rightarrow [S^1,Z]$ is trivial.

Let $x\in \pi_1(X)$ be a generator, $\langle x \rangle \cong \Z/p^n \subseteq \pi_1(X)$, we need to show that the composite $B\Z/p^n \rightarrow B\pi_1(X) \rightarrow B\pi_1(Z)$ is nullhomotopic for any generator $x$.

By hypothesis, there is a lift
$$
\xymatrix{
 & X \ar[r] \ar[d] & Z \ar[d] \\
B\Z/p^n \ar[ur] \ar[r] & B\pi_1(X) \ar[r] &  B\pi_1(Z).
}
$$

But since $Z$ is $B\Z/p$-null and $P_{B\Z/p}(B\Z/p^n)\simeq *$, it follows that the top composite $B\Z/p^n\rightarrow X \rightarrow Z$ is nullhomotopic, and therefore $\pi_1(f)(x)=0$.
%
%If we show that the composite $X\rightarrow Z_T\rightarrow BT$ is nullhomotopic then the map $X\rightarrow Z \rightarrow B\pi_1(Z)$ is also nullhomotopic.
%The map $X \rightarrow Z_T \rightarrow BT \rightarrow BT^{\wedge}_p$ is nullhomotopic by Lemma \ref{nullfamily} since it factors through $X\rightarrow Z_T \rightarrow (Z_T)^\wedge_p$. Since $\pi_1(BT^{\wedge}_p)=T/O^p(T)$ \cite[pages x-xi]{MR96c:55019}, there is a factorization $\pi_1(f_T)\colon \pi_1(X) \rightarrow O^p(T)\leq T$ which is an epimorphism since $\pi_(f_T)$ is an epimorphism. Therefore, this situation forces $T\cong O^p(T)$ to be a finite $p$-perfect group generated by elements of order $p$ which leads to a contradiction unless $T$ is the trivial group. Then $\pi_1(f)$ is the trivial morphism.
\end{proof}

The hypothesis in Theorem \ref{generalmainthmnull} concerning the pointed mapping space from the universal cover of $X$ is also satisfied by connected covers of $X$.

\begin{lemma}\label{covering} Let $X$ be a connected space.
\begin{enumerate}
\item Assume that $X$ is $1$-connected. Then $P_{B\Z/p}(X)^\wedge_p\simeq *$ iff $\map_*(X,Z)\simeq *$ for any connected $B\Z/p$-null $p$-complete space $Z$.
\item Assume that $X$ has a finite fundamental group. Let $Y$ be a connected cover of $X$. If $Z$ is a connected $B\Z/p$-null and $p$-complete space, then the equivalence $\map_*(Y,Z)\simeq *$ implies $\map_*(X,Z)\simeq *$.
\end{enumerate}
\end{lemma}

\begin{proof}
\begin{enumerate}
\item Note that if $Z$ is a connected $B\Z/p$-null $p$-complete space $Z$, there are weak homotopy equivalences $$\map_*(P_{B\Z/p}(X)^\wedge_p,Z)\simeq \map_*(P_{B\Z/p}(X),Z)\simeq \map(X,Z) $$ If $P_{B\Z/p}(X)^\wedge_p\simeq *$, then $\map(X,Z)\simeq *$ for any connected $B\Z/p$-null $p$-complete space $Z$. On the other hand, assume that $\map_*(X,Z)\simeq *$ for any connected $B\Z/p$-null $p$-complete space $Z$. Since $P_{B\Z/p}(X)^\wedge_p$ is a $p$-complete $B\Z/p$-null space by Corollary \ref{bastaenp}, $\map_*(P_{B\Z/p}(X)^\wedge_p, P_{B\Z/p}(X)^\wedge_p)\simeq *$, therefore $P_{B\Z/p}(X)^\wedge_p\simeq *$.
\item Consider a fibration $Y\rightarrow X \rightarrow BG$ where $G$ is a finite group. Zabrodsky's Lemma (see \cite[9.5]{Miller}) tells us that there is an equivalence of pointed mapping spaces $\map_*(X,Z)\simeq \map_*(BG,Z)$ since $\map_*(Y,Z)\simeq *$. Finally, this mapping space is contractible since we have weak homotopy equivalences $\map_*(BG,Z)\simeq \map_*(BG^\wedge_p,Z)$ and  $P_{B\Z/p}((BG)^\wedge_p)\simeq *$ by \cite[3.14]{Ramon}.
\end{enumerate}
\end{proof}

Now we are ready to undertake the proof of Theorem \ref{generalthmnull}.

\begin{proof}[Proof of Theorem \ref{generalthmnull}]
By hypothesis,  $\pi_1(X)$ has no quotients whose order is prime to
$p$, which amounts to say that $\pi_1(X)$ is equal to its $\mathbb{Z}/p$-radical $T_p(\pi_1(X))$.

First of all, notice that $L_{\mathbb{Z}[1/p]}X$ is $B\Z /p$-null by Lemma \ref{radicalperfect} and \cite[Lemma 6.2]{MR97i:55028}.
In order to show that $P_{B\Z/p}(X)\rightarrow L_{\Z[\frac{1}{p}]}(X)$ is a weak equivalence, since
$L_{\mathbb{Z}[1/p]}X$ is $B\Z /p$-null, we must show that for every $ B\mathbb{Z}/p$-null
space $Y$ the natural coaugmentation $X\longrightarrow
L_{\mathbb{Z}[1/p]}X$ gives a weak equivalence
$\map_*(L_{\mathbb{Z}[1/p]}X,Y)\simeq\map_*(X,Y)$.

%By the universal property of the
%localization, there is a weak equivalence of mapping spaces
%$\map_*(B\Z/p,L_{\mathbb{Z}[1/p]}(X))\simeq
%\map_*(L_{\mathbb{Z}[1/p]}(B\Z/p),L_{\mathbb{Z}[1/p]} BG)$,
%and the latter is contractible because the
%$\mathbb{Z}[1/p]$-homology of $ B\mathbb{Z}/p$ is trivial
%and $ B\mathbb{Z}/p$ is $\Z[1/p]$-good. Hence, the
%$\mathbb{Z}[1/p]$-localization of $X$ is a
%$ B\mathbb{Z}/p$-null space.

So let $Y$ be a $B\Z/p$-null space. Assume first that $Y$ is
simply-connected. By Miller's Theorem \ref{bastaenp-Miller}, $Y^\wedge_p$ is also $B\Z/p$-null.  According
to Bousfield-Kan fracture lemmas (\cite[V.6]{BK}), we must
prove that, for every prime $q$, there is a weak homotopy equivalence
$\map_*(L_{\mathbb{Z}[1/p]}(X),
Y_q^{\wedge})\simeq\map_*(X, Y_q^{\wedge})$, and
$\map_*(L_{\mathbb{Z}[1/p]}(X),
Y_{\mathbb{Q}})\simeq\map_*(X,
Y_{\mathbb{Q}})$. By  Lemmas \ref{univcoefficients} and \ref{covering}, this is a consequence of the assumption of the theorem, so we finish the situation in which $Y$ is simply connected.

Now let $Y$ be a $ B\mathbb{Z}/p$-null space and $\tilde{Y}$ its universal cover. The coaugmentation $X\longrightarrow L_{\mathbb{Z}[1/p]}
X$ induces a diagram of fibrations over the component of the constant map
$$
\xymatrix{
\map_*(L_{\mathbb{Z}[1/p]}(X),\tilde{Y})
\ar[r]^{\simeq} \ar[d] &
\map_*(X,\tilde{Y}) \ar[d] \\
\map_*(L_{\mathbb{Z}[1/p]}(X),Y)_{\{c\}} \ar[r]
\ar[d]^{\rho} & \map_*(X,Y)_{\{c\}}
\ar[d]^{\rho} \\
\map_*(L_{\mathbb{Z}[1/p]}(X), B\pi_1(Y))_c
\ar[r] & \map_*(X, B\pi_1(Y))_c }
$$
where $\map_*(L_{\mathbb{Z}[1/p]}(X),Y)_{\{c\}}$ and $\map_*(X,Y)_{\{c\}}$ are those components such that $\rho$ induce the constant map when composing with $Y\rightarrow B\pi_1(Y)$.

The top horizontal map is an equivalence because of the previous argument since $\tilde{Y}$ is a simply connected $B\Z/p$-null space. For any connected space $A$ and a discrete group $H$, $\map_*(A,BH)$ is a homotopically discrete space, and then $\map_*(L_{\mathbb{Z}[1/p]}(X), B\pi_1(Y))_c$ and  $\map_*(X, B\pi_1(Y))_c$ are contractible. Thus, the bottom horizontal arrow in the diagram is also a weak equivalence.

To finish the proof we need to show that there are weak homotopy equivalences  $\map_*(L_{\mathbb{Z}[1/p]}(X),Y)_{\{c\}}\simeq\map_*(L_{\mathbb{Z}[1/p]}(X),Y)$ and $\map_*(X,Y)_{\{c\}}\simeq \map_*(X,Y)$. The first equivalence follows now from the fact that $L_{\mathbb{Z}[1/p]}(X)$ is simply connected by Lemma \ref{radicalperfect}, while the second follows from Lemma \ref{fundamentalgroups}.
\end{proof}

\begin{proof}[Proof of Theorem \ref{generalmainthmnull}]
Theorem \ref{generalthmnull} applied to the universal cover of $X$ implies that the map in \cite[1.6]{MR97i:55028}, $P_{B\Z/p}(X\langle 1 \rangle)\rightarrow L_{\Z[\frac{1}{p}]}(X\langle 1 \rangle)$, is an equivalence.

Let $X_p$ be the covering space of $X$ with fundamental group $T_p(\pi_1(X))$. There is a fibration $X_p\rightarrow X \rightarrow B(\pi_1(X)/T_p(\pi_1(X)))$. Since the base space of this fibration $B(\pi_1(X)/T_p(\pi_1(X)))$ is $B\Z/p$-null, the nullification functor preserves the fibration by \cite[3.D.3]{Dror-Farjoun95} and there is another fibration $$P_{B\Z/p}(X_p)\rightarrow P_{B\Z/p}(X) \rightarrow B(\pi_1(X)/T_p(\pi_1(X))).$$ To prove the theorem we shall show that the natural map $P_{B\Z/p}(X_p)\rightarrow L_{\Z[\frac{1}{p}]}(X_p)$, which exists because $B\Z /p$ is $H\Z[\frac{1}{p}]$-acyclic and then $L_{\Z[\frac{1}{p}]}(X_p)$ is $B\Z /p$-null, is a homotopy equivalence. Note also that $(X_p)\langle 1 \rangle \simeq X\langle 1 \rangle$. Therefore $X_p$ also satisfies the hypothesis of the theorem.

From now on we assume that $\pi_1(X)$ has no quotients whose order is prime to
$p$, which amounts to say that $\pi_1(X)$ is equal to its $\mathbb{Z}/p$-radical $T_p(\pi_1(X))$.

Consider the fibration $X\langle 1 \rangle \rightarrow X \rightarrow B\pi_1(X)$ and its fibrewise nullfication (see \cite[1.F]{Dror-Farjoun95}) which gives a diagram of fibrations
$$
\xymatrix{
X\langle 1 \rangle \ar[r] \ar[d]^{\xi} &
X\ar[r] \ar[d]^{\bar{\xi}} & B\pi_1(X) \ar[d] \\
P_{B\Z/p}(X\langle 1 \rangle) \ar[r] & \bar{X} \ar[r] &  B\pi_1(X).
}
$$
where $\bar{\xi}$ is an equivalence after $B\Z /p$-nullification. Then, by \cite[1.6]{MR97i:55028}, $\bar{\xi}$ is a $\Z[\frac{1}{p}]$-equivalence. Note that it is enough to show that the map $P_{B\Z/p}(\bar{X})\rightarrow L_{\Z[\frac{1}{p}]}(\bar{X})$ is an equivalence since there is a chain
$$P_{B\Z/p}(X)\stackrel{\simeq}{\rightarrow} P_{B\Z/p}(\bar{X})\rightarrow L_{\Z[\frac{1}{p}]}(\bar{X})\stackrel{\simeq}{\leftarrow} L_{\Z[\frac{1}{p}]}({X}).$$

Moreover, $\pi_1(\bar{X})\cong \pi_1(X)$ because the fibre $P_{ B\Z/p}(X\langle 1 \rangle)$ is $1$-connected, then the universal cover of $\bar{X}$ is $P_{ B\Z/p}(X)$ and $\pi_1(\bar{X})=T_p(\pi_1(\bar{X}))$. For each generator $x\in \pi_1(\bar{X})$, the obstructions to lift the map $B(\langle x \rangle)\simeq B\Z/p^n\rightarrow B\pi_1(\bar{X})$ to $\bar{X}$ lie in the twisted cohomology groups $H^{i+1}(B\Z/p^n;\pi_i(P_{B\Z/p}(X\langle 1 \rangle)))$ for $i\geq1$, and these groups are trivial since the homotopy groups $\pi_i(P_{B\Z/p}(X\langle 1 \rangle))\cong \pi_i(L_{\Z[\frac{1}{p}]}(X\langle 1 \rangle))$ are $\Z[\frac{1}{p}]$-modules. That is, $\bar{X}$ is a connected space with finite fundamental group generated by $p$-torsion whose generators lift to $\bar{X}$.

In order to apply Theorem \ref{generalthmnull}, it remains to check that $\map_*(\bar{X}\langle 1 \rangle ,Z)\simeq *$ for any connected $B\Z/p$-null $p$-complete space $Z$. Recall that $\bar{X}\langle 1 \rangle\simeq P_{B\Z/p}({X}\langle 1 \rangle)$. Then $\map_*(P_{B\Z/p}(X\langle 1 \rangle),Z)\simeq \map_*({X}\langle 1 \rangle ,Z)\simeq *$ where the last equivalence follows by hypothesis.
%We apply Zabrodsky's Lemma (see \cite[9.5]{Miller}) to the fibration $P_{B\Z/p}(X\langle 1 \rangle)\rightarrow \bar{X} \rightarrow B\pi_1(X)$. Since $\map_*(P_{B\Z/p}(X\langle 1 \rangle),Z)\simeq \map_*({X}\langle 1 \rangle ,Z)\simeq *$ where the last equivalence follows by hypothesis, Zabrodsky's Lemma tells us that there is an equivalence $\map_*(\bar{X},X)\simeq \map_*(B\pi_1(X),Z)$. Finally, this mapping space is contractible since we have equivalences $\map_*(B\pi_1(X),Z)\simeq \map_*(B\pi_1(X)^\wedge_p,Z)$ and  $P_{B\Z/p}((B\pi_1(X))^\wedge_p)\simeq *$ by \cite[3.14]{Ramon}.
\end{proof}

\begin{remark}
The proof of Theorem \ref{generalmainthmnull} also holds if we replace the analysis at one prime $p$ for a set of primes $S$ and imposing that the hypothesis on pointed mapping spaces are satisfied for any prime $p$ in the set $S$. In that case we have to replace $L_{\Z[\frac{1}{p}]}$ by $L_{\Z[S^{-1}]}$, and $P_{B\Z/p}$ by $P_W$ where $W=\vee_{p\in S} B\Z/p$.
\end{remark}

\subsection{Examples}

We want to explore the implications of these results on classifying spaces of Lie groups, which was the original motivation for our work.
For this sake we need the following Lemma, which was proved by Dwyer
\cite[Theorem 1.2]{MR97i:55028} using an induction. We include here a
shorter proof, based on the homology decomposition of $ B G$ via
$p$-toral subgroups. The key point here is that this decomposition
is indexed over an mod $p$ \emph{acyclic} category, and this opens the way
for computing $P_{ B\Z /p}$ for a more general class of
$p$-good spaces (see Corollary \ref{plcgcoro}).
%We expect to follow this line of research in subsequent work.

\begin{lemma}
\label{Xnulo}
Let $Z$ be a connected $B\Z/p$-null and $p$-complete space. Let $F\colon \mathcal C \rightarrow Top$ be a functor such that for each object $c\in \mathcal C$, $F(c)$ is connected and $P_{B\Z/p}(F(c)^\wedge_p)$ is mod $p$ acyclic. If $|\mathcal C|^\wedge_p\simeq *$, then $\map_*(\hocolim_{\mathcal C} F(c),Z)\simeq *$.
\end{lemma}

\begin{proof}
The statement follows from a sequence of equivalences:
$$\map(\hocolim_{\mathcal C} F(c),Z)\simeq \holim_{\mathcal C} \map(F(c),Z)\simeq \holim_{\mathcal C} \map(P_{B\Z/p}(F(c)^\wedge_p),Z).$$
Under the hypothesis of the lemma, this last mapping space is homotopy equivalent to $\map(|\mathcal C|,Z)\simeq Z$ if  $|\mathcal C|^\wedge_p\simeq *$.
\end{proof}

\begin{cor}\label{exampleGLie}
Let $p$ be a prime. If $G$ is a compact Lie group and $X$ is a connected $p$-complete
$B\Z /p$-null space, then $\map_*(B
G^{\wedge}_p,X)$ is weakly contractible.
\end{cor}

\begin{proof}
The proof is divided into two steps. In the first one we assume that $G$ is a $p$-toral
group, and then we use the existence of mod $p$ homology decompositions of $BG$ with respect to certain families of $p$-toral subgroups of $G$, see \cite{JMO}, to undertake the general case.
%Recall that the
%$p$-completion induces an equivalence $\map_*( B
%G^{\wedge}_p,X)\simeq\map_*( B G,X)$ when $X$ is
%$p$-complete.
%\textbf{[Quiza la ultima frase es innecesaria, ya hemos usado esto varias veces en el paper, no?]}

Consider first when $G=T=(S^1)^n$. In this case, $ B
T^{\wedge}_p\simeq K(\Z^{\wedge}_p,2)^n\simeq ( B (\Z
/{p^{\infty}})^n)^{\wedge}_p$. As $X$ is a $p$-complete space, we have
the weak homotopy equivalence $\map_*( B T^{\wedge}_p,X)\simeq\map_*( B (\Z
/p^{\infty})^n ,X)$. This mapping space is contractible because $\Z
/p^{\infty}\cong\lim \Z/p^r$ defined by inclusions $\Z/p^n \subset \Z/p^{n+1}$, and since $\Z/p^r$ is a $p$-group, $B\Z/p^r$ is $B\Z/p$-acyclic and $\map_*(B\Z/p^r,X)\simeq *$, and we can apply Lemma \ref{Xnulo}.
Now, if $G=P$ is a $p$-toral group given by a group
extension $T^n\mono P\epi \pi$, Dwyer and Wilkerson show that $BG$ admits a $p$-discrete approximation \cite[Prop 6.9]{DW-annals}. There is a sequence of finite $p$-groups $P_0\subset P_1\subset \ldots$ such that $BP\simeq \hocolim BP_n$. Again, by Lemma \ref{Xnulo}, we obtain that $\map_*(BP,X)\simeq *$.

Let us go now through the general case. Our goal will be to prove that
the inclusion of constant maps induces an equivalence
$X\simeq\map( B G^{\wedge}_p,X)$. By work of
Jackowski-McClure-Oliver (\cite[Thm 4]{JMO}), the space $ B G$
is mod $p$ equivalent to $\hocolim_{\mathcal{O}_pG}F$,
where $\mathcal{O}_pG$ is the orbit category of stubborn $p$-toral subgroups of $G$
and $F$ is a functor whose values have the homotopy type of classifying spaces of stubborn $p$-toral subgroups of $G$. Since the statement holds for $p$-toral groups, by Lemma \ref{Xnulo} it is enough to observe that
$\mathcal{O}_pG$ is $\mathbb{F}_p$-acyclic, see \cite[Prop 6.1]{JMO}, and we are done.
\end{proof}

Now we are ready to prove the desired result, which was previously known for finite groups (\cite[Theorem 3.5]{Ramon}).

\begin{teo}
\label{teoremaprincipal}
Let $G$ be a compact Lie group and $\pi$ its group of components. Let $G_p$ be the subgroup of $G$ whose group of components is $T_{p}(\pi)$. Then the
$B\mathbb{Z}/p$-nullification of $BG$ fits in the
following covering fibration:
$$
L_{\mathbb{Z}[1/p]}BG_p\longrightarrow P_{B\mathbb{Z}/p}BG\longrightarrow B(\pi/\emph
{T}_{p}\pi)
$$
\end{teo}

\begin{proof}
We have to check that the assumptions on Theorem \ref{generalmainthmnull} are satisfied when $X=BG$. Since the universal cover of $BG$ is $BG_0$ is again the classifying space of a compact Lie group, by Corollary \ref{exampleGLie} the hypothesis of  Theorem \ref{generalthmnull} are satisfied.
\end{proof}

The proof of Corollary \ref{exampleGLie} applies to other type of spaces which admits mod $p$ homology decompositions, for example $p$-compact groups (see \cite{CLN}). The theory of $p$-local compact groups introduced by Broto, Levi and Oliver in \cite{MR2302494} includes both the theory of $p$-compact groups \cite{DW-annals} and $p$-local finite groups \cite{BLO2}. Roughly speaking, a $p$-local compact is a triple $(S,\mathcal F,\mathcal L)$ where $S$ is a discrete $p$-toral group and $\mathcal F$ and $\mathcal L$ are categories which model conjugacy relations among subgroups of $S$. The classifying space of a $p$-local compact group is $|\mathcal L|^\wedge_p$, and one of the main features of $p$-local compact groups is that this space admits mod $p$-homology decompositions in terms of classifying spaces of $p$-compact toral subgroups over mod $p$ acyclic orbit categories (see \cite[Proposition 4.6]{MR2302494} and \cite[Corollary 5.6]{MR2302494}).

\begin{prop}\label{plcgcoro}
Let $p$ be a prime, and $(S,\mathcal F,\mathcal L)$ a $p$-local compact group. Then there is an equivalence $L_{\mathbb{Z}[1/p]}(|\mathcal L|^\wedge_p)\simeq {P}_{\emph{B}\mathbb{Z}/p}(|\mathcal L|^\wedge_p)$.
\end{prop}

\begin{proof}
First of all, $\pi_1(|\mathcal L|^\wedge_p)$ is a finite $p$-group by \cite[Proposition 4.4]{MR2302494}, therefore $T_p(\pi_1(|\mathcal L|^\wedge_p))=\pi_1(|\mathcal L|^\wedge_p)$. Then, we only need to check the hypothesis in Theorem \ref{generalthmnull}. That is, $\map_*(X\langle 1 \rangle,Z)\simeq *$ for any connected $B\Z/p$-null $p$-complete space $Z$.

The same argument used in the proof of Corollary \ref{exampleGLie} using mod $p$ homology decompositions can be applied and it shows that $\map_*(|\mathcal L|^\wedge_p,Z)\simeq *$ for any connected $B\Z/p$-null $p$-complete space $Z$. But it is not known in general if the universal cover $|\mathcal L|^\wedge_p\langle 1 \rangle$ is the classifying space of a $p$-local compact group. Instead, we will check that the proof of Corollary \ref{exampleGLie} applies by showing that $|\mathcal L|^\wedge_p\langle 1 \rangle$ admits a description, up to $p$-completion, as a homotopy colimit of $B\Z/p$-acyclic spaces over a mod $p$-acyclic category.

Let $P\leq S$ be an object in $\mathcal O (\mathcal F_0)$ (see \cite[Proposition 4.6]{MR2302494}) and let $E_{P\leq S}$ be the pullback of $|\mathcal L|^\wedge_p\langle 1 \rangle\rightarrow |\mathcal L|^\wedge_p$ along $\tilde{B}P\rightarrow BS$. Then, by naturality there is a map $\hocolim_{\mathcal O (\mathcal F_0) } E_{P\leq S}\rightarrow |\mathcal L|^\wedge_p\langle 1 \rangle$ which fits in a diagram of fibrations by Puppe's theorem (e.g. \cite[Appendix]{Dror-Farjoun95}),
$$
\xymatrix{
|\mathcal L|^\wedge_p\langle 1 \rangle \ar[r]  &
|\mathcal L|^\wedge_p\ar[r]  & B\pi_1(|\mathcal L|^\wedge_p) \\
\hocolim_{\mathcal O (\mathcal F_0) } E_{P\leq S} \ar[r] \ar[u]& \hocolim_{\mathcal O (\mathcal F_0) } \tilde{B}(P) \ar[r] \ar[u] &  B\pi_1(|\mathcal L|^\wedge_p)\ar[u]^{id}.
}
$$

Since the middle vertical arrow is a mod $p$-equivalence, it follows that the left vertical arrow is also a mod $p$-equivalence. Moreover, $(\tilde{B}P)^\wedge_p$ is the classifying space of a $p$-compact toral group, so it follows from the fibration $E_{P\leq S}\rightarrow \tilde{B}P\rightarrow B\pi_1(|\mathcal L|^\wedge_p)$ that $(E_{P\leq S})^\wedge_p$ is also the classifying space of a $p$-compact toral group, and therefore $B\Z/p$-acyclic. Then the proof of Corollary \ref{exampleGLie} applies.
\end{proof}

A finite loop space is a triple $(X,BX,e)$ where $e\colon X\rightarrow \Omega BX$ is a weak equivalence and $H^*(X;\Z)$ is finite. Note that if $X$ is a finite loop space, then $BX\langle 1 \rangle ^\wedge_p$ is the classifying space of connected $p$-compact group since $H^*(\Omega BX\langle 1\rangle;\F_p)$ is finite dimensional (see \cite{DW-annals}).

\begin{cor}\label{finiteloopspaces}
Let $X$ be a finite loop space. Then there is a fibration $$
L_{\mathbb{Z}[1/p]}BX_p\longrightarrow P_{B\mathbb{Z}/p}BX\longrightarrow B(\pi/\emph
{T}_{p}\pi)
$$ where $\pi=\pi_0(X)$ and $BX_p$ is the covering of $X$ whose fundamental group is $T_p(\pi_0(X))$.
\end{cor}

\begin{proof}
Since $BX\langle 1 \rangle ^\wedge_p$ is the classifying space of a $p$-compact group, Proposition \ref{plcgcoro} shows that $P_{B\Z/p}(BX\langle 1 \rangle )^\wedge_p\simeq *$. Therefore we can apply Theorem \ref{generalmainthmnull}.
\end{proof}

We finish with a somewhat two different examples.

\begin{cor}\label{KacMoodygroups}
Let $p$ be a prime. Let $K$ be a Ka\v{c}-Moody group with a finite group of components. Then there is a fibration
$$
L_{\mathbb{Z}[1/p]}BK_p\longrightarrow P_{\emph{B}\mathbb{Z}/p}BK \longrightarrow\emph{B}(\pi_1(BK)/
{T}_{p}(\pi_1(BK))
$$
where $K_p$ is the subgroup of $K$ whose group of components is $T_{p}(\pi_1(BK)$
\end{cor}

\begin{proof}
First of all, the universal cover of $BK$ is $BK_0$ where $K_0$ is the connected component of the unit in $K$. It is shown in Nitu Kitchloo thesis (see \cite{Broto-Kitchloo}) that $BK$ is homotopy equivalent to a colimit over a contractible category of classifying spaces of compact Lie groups. Then Corollary \ref{exampleGLie} and its proof apply to show that $\map_*(BK,Z)\simeq *$ for any connected $B\Z/p$-null $p$-complete space $Z$.
\end{proof}

A different kind of example arises from the theory of infinite loop spaces and it is a consequence of Theorem $2$ in \cite{McGibbon}.

\begin{cor}\label{infiniteloopspaces}
Let $X$ be a connected infinite loop space  with finite fundamental group. Then there is a fibration
$$L_{\Z[\frac{1}{p}]}(X_p)\rightarrow P_{B\Z/p}(X)\rightarrow B(\pi_1(X)/T_p(\pi_1(X)))$$ where $X_p$ is the covering of $X$ whose fundamental group is $T_p(\pi_1(X))$.
\end{cor}

%%%%%%%%%%%%%comentado%%%%%%%%%%%%%

%%%%%%%%%%%%%%%%%%%%%
\section{Relation of nullification functors with other idempotent functors}
%%%%%%%%%%%%%%%%%%%%%

In this section we compare the effect of nullification
$P_{ B\Z /p}$ on spaces which satisfy the hypothesis of Theorem \ref{generalmainthmnull} with the effect of some completions or localizations on it.
We analyse both functors that are supposed to kill the
$p$-torsion, like $L_{\Z [1/p]}$ or $\Z [1/p]_{\infty}$, and
functors that usually preserve it, as $L_{\Z [1/q]}$ and
$p$-completion do.

%It is worth to recall here that if a space $X$ is not $R$-good for
%a ring $R$, then $R_{\infty}X$ can be different from $L_RX$; in
%particular, if $R=\Z [1/p]$ and $X= BG$, we cannot assume
%that $\Z [1/p]_{\infty} BG$ has the homotopy type of
%$L_{\Z [1/p]} B G$ unless $\pi_0G$ is $\Z [1/p]$-perfect. See
%\cite{BK} for details.

\begin{lemma}
\label{pq}
Let $X$ be a connected space with finite fundamental group, $p$ and $q$
different primes. Then $(X^{\wedge}_p)^{\wedge}_q$ is contractible.
\end{lemma}

\begin{proof}

If $X$ is $1$-connected the case of $\mathbb{F}_q$-completion is described
in \cite[VI.5.1]{BK}. If $X$ is not simply-connected, consider the fibration $X^\wedge_p \langle 1 \rangle \rightarrow X^\wedge_p \rightarrow B\pi_1(X^\wedge_p)$ and its fibrewise $q$-completion, $$(X^\wedge_p \langle 1\rangle)^\wedge_q\rightarrow Y\rightarrow B\pi_1(X^\wedge_p).$$

Since the fibre is a $1$-connected $p$-complete space completed at $q$ it is contractible, then $Y\simeq B\pi_1(X^\wedge_p)$. But then $(X^{\wedge}_p)^{\wedge}_q\simeq Y^\wedge_q\simeq B\pi_1(X^\wedge_p)^\wedge_q$ which is contractible since $B\pi_1(X^\wedge_p)$ is the classifying space of a finite $p$-group.
\end{proof}

We start by showing some direct direct consequences of Theorem \ref{generalmainthmnull}.

\begin{prop}\label{pycompletion}
Let $X$ be a space which satisfies the hypothesis of Theorem \ref{generalmainthmnull}. Then $\pi_1(P_{B\Z/p}(X))=\pi_1(X)/T_p(\pi_1(X))$ and $(P_{B\Z/p}(X))^\wedge_p\simeq {*}$. Moreover $P_{B\Z/p}(X^\wedge_p)\simeq L_{\Z[\frac{1}{p}]}(X^\wedge_p)\simeq (X^\wedge_p)_\Q$ is $1$-connected.
\end{prop}

\begin{proof}
Since $L_{\Z[\frac{1}{p}]}(X_p)$ is $1$-connected by Lemma \ref{radicalperfect}, it is clear from the fibration in Theorem \ref{generalmainthmnull} that $\pi_1(P_{B\Z/p}(X))=\pi_1(X)/T_p(\pi_1(X))$.
The space $L_{\Z[\frac{1}{p}]}(X_p)$ is mod $p$ acyclic by \cite[Lemma 6.2]{MR97i:55028} and the order of $\pi_1(X)/T_p(\pi_1(X))$ is prime to $p$, it follows that $(P_{B\Z/p}(X))^\wedge_p$ is weakly contractible.

The second statement follows from applying Theorem \ref{generalmainthmnull} to $X^\wedge_p$. Observe that if $X$ satisfies the hypothesis of the theorem, then $X^\wedge_p$ also does. Moreover, $\pi_1(X^\wedge_p)$ is a finite $p$-group, then   $P_{B\Z/p}(X^\wedge_p)\simeq L_{\Z[\frac{1}{p}]}(X^\wedge_p)$. It remains to prove that they are equivalent to $(X^\wedge_p)_\Q$. Since they are $1$-connected we can apply Sullivan's arithmetic square.

We have proved that  $(P_{B\Z/p}(X^\wedge_p))^\wedge_p$ is weakly contractible. Moreover, if $q\neq p$ then $(P_{B\Z/p}(X^\wedge_p))^\wedge_q\simeq (X^\wedge_p)^\wedge_q$ which is weakly contractible by Lemma \ref{pq}. Then $P_{B\Z/p}(X^\wedge_p) \simeq P_{B\Z/p}(X^\wedge_p)_\Q\simeq  (X^\wedge_p)_\Q$.
\end{proof}

We start by showing that $ B\Z /p$-nullification and
$p$-completion behave like opposite functors in this context.

\begin{remark} If we complete in one prime $q$ and $ B\Z /p$-nullify with
regard to a different prime $p$, then $X^\wedge_q$ is $ B\Z /p$-null
and the coaugmentation $X\ra P_{ B\Z /p}X$ is an
equivalence after $q$-completion if $X$ satisfies the hypothesis of Lemma \ref{null+comp}.
\end{remark}

\begin{remark}
Note that in general a connected space $X$ could be $\Z [1/p]$-bad if $X$
is not $1$-connected, and then it is not possible in general to
replace completion by localization in the previous results. If we know in
advance that $X$ is $\Z [1/p]$-good (this happens, for
example, if its fundamental group is $\Z [1/p]$-perfect) then we
can do the replacement, and moreover $\Z [1/p]_{\infty}X\simeq
L_{\Z [1/p]}X$. See for example (\cite[1.E]{Dror-Farjoun95}) for
more information about the relation between $R$-localization and
$R$-completion.
\label{idemp}
\end{remark}

\begin{prop}
\label{pqPcompletion}
Let $X$ be a space which satisfies the hypothesis of Theorem \ref{generalmainthmnull} and such that $\pi_1(X)\cong T_q(\pi_1(X))$. Then there are homotopy equivalences
$$P_{B\mathbb{Z}/p}L_{\mathbb{Z}[1/q]}X\simeq L_{\mathbb{Z}[1/p,1/q]}X\simeq
L_{\mathbb{Z}[1/p]}P_{B\mathbb{Z}/q}X.$$
\end{prop}

\begin{proof}
Since $\pi_1(X)\cong T_q(\pi_1(X))$, $L_{\Z[\frac{1}{q}]}(X)\simeq P_{B\Z/q}(X)$ is $1$-connected. Then, by Theorem  \ref{generalmainthmnull}, we have $P_{B\Z/p}(P_{B\Z/q}(X))\simeq P_{B\Z/p}(L_{\Z[\frac{1}{q}]}(X))\simeq L_{\Z[\frac{1}{p}]}(L_{\Z[\frac{1}{q}]}(X))\simeq  L_{\Z[\frac{1}{p}]}(P_{B\Z/q}(X))$.
\end{proof}

We finish by establishing the commutativity of the functors
$P_{ B\Z /p}$ and $P_{ B\Z /q}$.
The problem of commutation of localization functors was
extensively studied in \cite{RS}.

\begin{prop}
\label{commutativity}
Let $X$ be a connected space, $p$ and $q$ two different
primes. Assume that $X$ satisfies the hypothesis of Theorem \ref{generalmainthmnull} for both primes $p$ and $q$.
Then there are homotopy equivalences
$$P_{B\mathbb{Z}/p}P_{B\mathbb{Z}/q}X
\simeq P_{B\mathbb{Z}/p\vee B\mathbb{Z}/q}X\simeq
P_{B\mathbb{Z}/q}P_{B\mathbb{Z}/p}X.$$
\end{prop}

\begin{proof}
It is enough to show the first equivalence since the other one will follow by symmetry.

Consider the set of primes $S=\{p,q\}$. By pulling back the universal fibration, there is a fibration $X_S\rightarrow X \rightarrow B(\pi_1(X)/(T_S(\pi_1(X)))$, where $\pi_1(X_S)=T_S(\pi_1(X))$. Since the order of $\pi_1(X)/T_S(\pi_1(X))$ is prime to both $p$ and $q$, the space $B(\pi_1(X)/(T_S(\pi_1(X)))$ is both $B\Z/p$-null and $B\Z/q$-null (in particular it is also $B\Z/p\vee B\Z/q$-null) the composite of functors $P_{B\Z/p}\circ P_{B\Z/q}$ and $P_{B\Z/p\vee B\Z/q}$ preserve the fibration \cite[3.D.3]{Dror-Farjoun95}, and there is a diagram of fibrations
$$
\xymatrix{
P_{B\Z/p}( P_{B\Z/q}(X_S)) \ar[r] \ar[d] & P_{B\Z/p}( P_{B\Z/q}(X)) \ar[r] \ar[d]  & B(\pi_1(X)/(T_S(\pi_1(X))) \ar[d]^{id} \\
P_{B\Z/p\vee B\Z/q}(X_S) \ar[r] & P_{B\Z/p\vee B\Z/q}(X) \ar[r] & B(\pi_1(X)/(T_S(\pi_1(X))).
}
$$
where the first two vertical maps exist because if a space $Y$ is $B\Z/p\vee B\Z/q$-null, and then it is also $B\Z/p$-null and $B\Z/q$-null. Then we can assume that the group $\pi_1(X)=T_S(\pi_1(X))$, which we simply denote by $\pi$ in the sequel, is generated by $p$ and $q$ torsion.

By \cite[Prop 1.1]{RS}, we need to show that  $P_{B\Z/p}( P_{B\Z/q}(X))$ is $B\Z/q$-null and conversely $P_{B\Z/q}( P_{B\Z/p}(X))$ is $B\Z/p$-null. In our situation, by symmetry,  it is enough to check one of the two conditions.

Let us see first that $P_{B\Z/p}( P_{B\Z/q}(X))$ is $1$-connected. We can apply the fibrewise $B\Z/p$-nullification to the fibration in Theorem \ref{generalmainthmnull},
$$
\xymatrix{
L_{\Z[\frac{1}{q}]}(X_q) \ar[r] \ar[d] & P_{B\Z/q}(X_q) \ar[r] \ar[d]  & B(\pi/T_q(\pi)) \ar[d]^{id} \\
P_{B\Z/p}(L_{\Z[\frac{1}{q}]}(X_q)) \ar[r] & \bar{P} \ar[r] & B(\pi/T_q(\pi)),
}
$$
where $L_{\Z[\frac{1}{q}]}(X_q)$ is $1$-connected and $P_{B\Z/p}(\bar{P})\simeq P_{B\Z/p}(P_{B\Z/q}(X_q))$;\v{} then $\bar{P}$ has fundamental group $\pi/T_q(\pi)$ which is generated by $p$ torsion. If $\bar{P}$ satisfies the hypothesis of Theorem \ref{generalmainthmnull} for the prime $p$ then $P_{B\Z/p}(\bar{P})\simeq L_{\Z[\frac{1}{p}]}(\bar{P})$ is $1$-connected. We need to check that for any connected space $Z$ which is $p$-complete and $B\Z/p$-null, $\map_*(\bar{P}\langle 1 \rangle,Z)$ is weakly contractible. Note that $\bar{P}\langle 1 \rangle\simeq P_{B\Z/p}(L_{\Z[\frac{1}{q}]}(X_q))$, and then
$$\textrm{map}_*(P_{B\Z/p}(L_{\Z[\frac{1}{q}]}(X_q)),Z)\simeq \textrm{map}_*(L_{\Z[\frac{1}{q}]}(X_q),Z)\simeq \textrm{map}_*(P_{B\Z/q}(X_q),Z)\simeq$$ $$\simeq \textrm{map}_*(X_q,Z),$$
where the last equivalence follows because $(P_{B\Z/q}(X_q))^\wedge_p\simeq (X_q)^\wedge_p$ and $Z$ is $p$-complete. Finally Lemma \ref{covering} tells us that this last mapping space is weakly contractible.

We denote by $Y$ the space $P_{B\Z/p}( P_{B\Z/q}(X))$, and we finally should check that it is $B\Z/q$-null. Since it is a $1$-connected space, we can use Sullivan's arithmetic square and check that the mapping spaces $\map_*(B\Z/q, Y_\Q)$ and $\map_*(B\Z/q, Y^\wedge_r)$ are weakly contractible for any prime $r$.

If $r\neq q$, $\map_*(B\Z/q, Y^\wedge_r)\simeq *$ because $(B\Z/q)^\wedge_r\simeq *$. Also, since $(B\Z/q)_\Q\simeq *$,  $\map_*(B\Z/q, Y_\Q)\simeq *$. We are left to the case $r=q$, and $Y^\wedge_q=(P_{B\Z/p}( P_{B\Z/q}(X)))^\wedge_q\simeq (P_{B\Z/q}(X))^\wedge_q\simeq *$ by Proposition \ref{pycompletion}. So we are done.
\end{proof}

For example, given a compact Lie group, $BG$ satisfies the hypothesis of Theorem \ref{generalmainthmnull} for any prime $p$.

\begin{remark}
The same proof remains valid if we apply in succession over $X$
a finite number of $ B\Z /p$-nullification functors for different
primes assuming $X$ satisifes the hypothesis of Theorem \ref{generalmainthmnull} for each prime.
On the other hand, it is likely that that the
nullification of $X$ with regard to the wedge of the
classifying spaces of \emph{all} primes is homotopy equivalent to
the rational localization of $X$. See \cite[Section 3.2]{Ramon} for details.
\end{remark}

%%%%%%%%%%%%%%%%%%%%%%%%%%%%%%
\section{$B\Z/p$-cellularization of classifying spaces}
%%%%%%%%%%%%%%%%%%%%%%%%%%%%%%

In this section we will give a Serre-type general dichotomy theorem (Theorem \ref{thmdichotomy}), which is very much in the spirit of \cite{FS07}. Then, we will use this statement to describe several examples concerning the $B\Z /p$-cellularization of some families of classifying spaces of remarkable groups, such as $p$-toral groups, finite groups with a $p$-subgroup of $p'$-index, $BS^3$ or $BSO(3)$ (at the prime $2$). Our considerations are also based in the results of the previous sections relating cellularization, nullification and completion.

\subsection{The dichotomy theorem}

The main result of this section is:

\begin{teo}\label{thmdichotomy}
%%%%%%%%%%%%%%%%%%%%%%%%%%%%%%%%%%%%%%%%%%%%%%%%%%%%%%%%%%%
 Let $X$ be a connected nilpotent $\Sigma^n B\Z/p$-null space for some $n\geq 0$. Then the $B\Z /p$-cellullarization of $X$ has the homotopy type of a Postnikov piece with homotopy groups are concentrated in degrees $1$ to $n$ , or else it has infinitely many nonzero homotopy groups. Moreover, if $X$ is $1$-connected of finite type, then the fundamental group $\pi_1(CW_{B\Z/p}(X))$ is a finite elementary abelian $p$-group.
 \end{teo}

 The rest of the subsection is devoted to the proof of Theorem \ref{thmdichotomy}. Even if the statement is similar to the one in \cite[Proposition 2.3]{FS07}, the authors deal with the situation in which the space is torsion, and this is not the case for $BG$ where $G$ is a compact connected Lie group. The strategy used  in \cite{FS07} for classifying spaces of finite groups can be summarized as follows.

\begin{prop}
Let $X$ be a torsion Postnikov piece whose fundamental group is generated by elements of order $p$ which lift to $X$. Assume there exists a prime $q\neq p$ such that $X^\wedge_q$ is torsion and it has infinitely many non-trivial homotopy groups. Then $CW_{B\Z/p}(X)$ also has infinitely many non-trivial homotopy groups.
\end{prop}

\begin{proof}
Consider the fibration $CW_{B\Z/p}(X)\rightarrow X\rightarrow P_{\Sigma B\Z/p}(C)$ from Theorem \ref{Chacho}. Note that $P_{\Sigma B\Z/p}(C)$ is $1$-connected since $C$ is so. To prove the statement, we will show that  $P_{\Sigma B\Z/p}(C)$ has infinitely many non-trivial homotopy groups. We apply Sullivan's arithmetic square to $P_{\Sigma B\Z/p}(C)$ to obtain a pullback diagram
$$\xymatrix{
P_{\Sigma B\Z/p}(C) \ar[r] \ar[d]  & (\prod_{r \neq p} X^\wedge_r) \times (P_{\Sigma B\Z/p}(C))^\wedge_p \ar[d] \\
\textrm{*} \ar[r] & (\prod_{r\neq q,p} X^\wedge_r)_\Q \times ((P_{\Sigma B\Z/p}(C))^\wedge_p)_\Q.
}$$
which allow us to construct a map $s\colon X^\wedge_q\rightarrow P_{\Sigma B\Z/p}(C)$ such that $\eta\circ s\simeq id$.  That is $s$ is a section of the $q$-completion. Then for $n\geq 2$ we have that $\pi_n(X^\wedge_q)$ is a direct summand of $\pi_n(P_{\Sigma B\Z/p}(C))$.
\end{proof}

For example, by Levi's work in \cite{MR96c:55019},
the previous theorem applies when $X$ is the classifying space of a finite group.

Now we need to state some general results concerning to the cellularization of $\Sigma B\Z/p$-null spaces, that deal with the consequences of imposing that $CW_{B\Z/p}(X)$ is a Postnikov piece for a certain space $X$. Note that this is the ``forbidden" case in Theorem \ref{thmdichotomy}.

\begin{lemma}\label{ngeq3}
Let $P[n]$ be a connected $\Sigma B\Z/p$-null Postnikov piece with $n\geq 3$ then $p$ is invertible in $\pi_n(P[n])$.
\end{lemma}

\begin{proof}
Note that if $P[n]$ is $\Sigma B\Z/p$-null then $\Omega P[n]$ is $B\Z/p$-null, and also $\Omega^{n-1}P[n]$ is so. Since the connected component of the constant in $\Omega^{n-1}P[n]$ is an Eilenberg-MacLane space $K(\pi_n(P[n]), 1)$, we see that $K(\pi_n(P[n]), 1)$ is also $B\Z/p$-null.
%In particular, $Hom(\Z/p, \pi_n(P[n]))\cong[B\Z/p, K(\pi_n(P[n]),1)]=*$, and $\pi_n(P[n])$ has no $p$-torsion.

Similarly, the connected component $E$ of $\Omega^{n-2}P[n]$ is $B\Z/p$-null. There is a fibration $K(\pi_n(P[n]),2)\rightarrow E \rightarrow K(\pi_{n-1}(P[n]),1)$. Since the pointed mapping spaces $\map_*(B\Z/p, K(\pi_{n-1}(P[n]),1))_c$ and $\map_*(B\Z/p, E)$ are weakly contractible, we obtain from the previous fibration that $\map(B\Z/p, K(\pi_n(P[n]),2))$ is also weakly contractible. That is $K(\pi_n(P[n]),2)$ is $B\Z/p$-null. The conclusion now follows from Lemma \ref{nullK(G,n)}.
\end{proof}

\begin{prop}\label{general-dichotomy}
Let $X$ be a connected nilpotent $\Sigma B\Z/p$-null space such that $\Z[\frac{1}{p}]_\infty (X)\simeq *$. Then $X$ has the homotopy type of a $K(G,1)$ or it has infinitely many nonzero homotopy groups.
\end{prop}

\begin{proof}%[Proof of Proposition \ref{general-dichotomy}]
Assume that $X\simeq P[n]$ is a finite Postnikov piece. First we show that $n\leq 2$. Since $\Z[\frac{1}{p}]_\infty(X)$ is weakly contractible (see Lemma \ref{acyclic-cellularization}), then $\pi_i(P[n])\otimes \Z[\frac{1}{p}]=0$ for all $i>0$ (\cite[V.4.1]{BK}). But if $n\geq 3$, $p$ is invertible in $\pi_n(P[n])$ by the previous Proposition \ref{ngeq3}, then $\pi_n(P[n])=0$. We can apply this argument as long as $n\geq 3$.

From now on we assume $n=2$. Next we prove that $\pi_2(P[2])$ has no $p$-torsion. There is a fibration $K(\pi_2(P[2]),2)\rightarrow P[2] \rightarrow K(\pi_1(P[2]),1)$, which induces a covering $K(\pi_1(P[2]),0)\rightarrow K(\pi_2(P[2]),2)\rightarrow P[2]$. If $\pi_2(P[2])$ had $p$-torsion, then a non-trivial homomorphism $\Z/p\rightarrow \pi_2(P[2])$ would induce a nontrivial map $f\colon \Sigma B\Z/p\rightarrow K(\pi_2(P[2]),2)$ which is nullhomotopic when composed with  $K(\pi_2(P[2]),2)\rightarrow P[2]$ since $P[2]$ is $\Sigma B\Z/p$-null (see Lemma \ref{suspension-null}). Then we obtain a contradiction since $f$ must be nullhomotopic. So $\pi_2(P[2])$ has no $p$-torsion, and this is again a contradiction since $\pi_2(P[2])\otimes \Z[\frac{1}{p}]=0$.
\end{proof}

We now state  our dichotomy theorem for nilpotent $\Sigma B\Z/p$-null spaces.

\begin{teo}\label{thmdichotomy1}
Let $X$ be a connected nilpotent $\Sigma B\Z/p$-null space. Then the $B\Z /p$-cellulla-rization of $X$ has the homotopy type of a $K(G,1)$ or it has infinitely many nontrivial homotopy groups. Moreover, if $X$ is $1$-connected of finite type, then $\pi_1(CW_{B\Z/p}(X))$ is a finite elementary abelian $p$-group.
\end{teo}

\begin{proof}
By Lemma \ref{nilpotent-cellularization}, $CW_{B\Z/p}(X)$ is also nilpotent. Moreover,  by Lemma \ref{acyclic-cellularization} we have an equivalence $\Z[\frac{1}{p}]_\infty (CW_{B\Z/p}(X))\simeq *$ , and then we can apply Proposition \ref{general-dichotomy} to $CW_{B\Z/p}(X)$.
\end{proof}

\begin{proof}[Proof of Theorem \ref{thmdichotomy}]
Assume that $n\geq 2$, we can assume that $X$ is $B\Z/p$-cellular and  $\Sigma^n B\Z/p$-null. By \cite[Theorem 7.2]{B2}, there is a principal fibration $$K(P,n)\rightarrow X\rightarrow P_{\Sigma^{n-1}}(X)$$ where $P$ is a $p$-torsion grup. Since $K(P,n)$ is $B\Z/p$-cellular (\cite[Lemma 3.3]{CCS2}) and $X$ too by assumption, then  $P_{\Sigma^{n-1}}(X)$ is also $B\Z/p$-cellular by \cite[Theorem 4.7]{Chacholski96}. By induction we reduce to the case in which $n=1$, which is proved in Theorem \ref{thmdichotomy1}.
\end{proof}

The next question we need to refer concerning Theorem \ref{thmdichotomy} is when the cellularization of a classifying space is again a classifying space, not necessarily of a discrete group. This is important to understand the first part of the previous dichotomy.

\begin{prop}
\label{Nocellular}
Let  $X$ be a space. If $CW_{B\Z/p}(X)\simeq BH$ for
some compact Lie group $H$, then it must be a finite $p$-group
generated by order $p$ elements.
\end{prop}

\begin{proof}
Since the pointed homotopy colimit of acyclics is acyclic for any
cohomology theory (\cite[2.D.2.5]{Dror-Farjoun95}), it is clear
that $\tilde{H}^*( B H;\Q)=\tilde{H}^*(CW_{ B\Z /p}X;\mathbb{Q})=0$. On the other hand, it is well-known that the
rational cohomology of $ B H$ are the invariants of the rational
cohomology of the classifying space of the maximal torus $T$ under
the action of the Weyl group $W$. In fact $\tilde{H}^*( B H;\Q)=0$ if and only if $H$ is a finite group.
 Finally, the functor $CW$ is idempotent,
so $ B H$ must be $ B\Z /p$-cellular. Thus, we can apply
\cite[Prop 4.14 ]{Ramon} to finish the proof.
\end{proof}

\begin{remark}
The arguments in Proposition \ref{Nocellular} also work if $CW_{B\Z/p}(X)\simeq (BH)^{\wedge}_p$ where $H$ is a compact Lie group. It is clear then that $\tilde{H}^*(BH;\Z^{\wedge}_p)\otimes \Q=0$. But again this is only possible if $H$ is discrete. If $H$ is in particular finite, conditions are known (see \cite[Corollary 3.3]{FS07}) under which $(BH)^{\wedge}_p$ is $B\Z /p$-cellular. See Example \ref{rorder} below.
\end{remark}

\begin{remark}
When $X$ is an $H$-space satisfying the hypothesis of Theorem \ref{thmdichotomy}, Castellana, Crespo and Scherer proved in \cite{CCS2} that the $B\Z/p$-cellularization of $X$ is always a Postnikov piece. Examples of such spaces are given by $H$-spaces whose mod $p$ cohomology is finitely generated as an algebra over the Steenrod algebra (see \cite{CCS1}).
\end{remark}

\subsection{Examples}
In this subsection we concentrate in the description of the $B\Z /p$-cellularization of classifying spaces of compact Lie groups, generalizing to the continuous case work of the second author in the finite case (\cite{Ramon} and \cite{FS07}). In the study of the homotopy type of classifying spaces of Lie groups, a very useful strategy is to isolate the information at every prime.

Theorem \ref{thmdichotomy} implies automatically the following dichotomy theorem for classifying spaces of compact Lie groups. We say that an element $g\in G$ is $p$-cohomologically central if the map induced by the inclusion $BC_G(x)\rightarrow BG$ is a mod $p$ homology isomorphism. Mislin in \cite{Mislin} shows that there is a natural bijection between the set of conjugacy classes of $p$-cohomologically central elements of order $p$ in $G$ with $pZ(G/O_{p'}(G))$ where $O_{p'}(G)$ is the largest normal $p'$-subgroup of $G$ and $pZ(G)$ are the elements of order $p$ in $Z(G)$.

\begin{teo}
\label{thmdichotomyBG}
Let $G$ be a compact connected Lie group. If there exists a non $p$-coho\-mo\-lo\-gi\-cally central element of order $p$, then  the $B\Z /p$-cellullarization of $BG$ has infinitely many nonzero homotopy groups. Otherwise, it has the homotopy type of a $K(V,1)$, where $V$ is a finite elementary abelian $p$-group.
\end{teo}

\begin{proof}
Since $G$ is assumed to be connected, $BG$ is simply connected. Moreover it is of finite type, and $\Sigma B\Z /p$-null because of Miller's solution of the Sullivan conjecture because $\Omega BG\simeq G$ is a finite complex. Now, we apply Theorem \ref{thmdichotomy}. If $CW_{BZ/p}(BG)$ is an Eilenberg-MacLane space $K(V,1)$, then $\map_*(B\Z/p,BG)$ is homotopically discrete. Since $BG$ is simply connected, $[B\Z/p,BG]_*\cong [B\Z/p,BG]\cong \Rep(\Z/p,G)$. If $\map_*(B\Z/p,BG)$ is homotopically discrete, then for each $\rho \in \Rep(\Z/p,G)$, the evaluation $\map(B\Z/p,BG)_{B\rho}\rightarrow BG$ induces an $\F_p$-homology equivalence $BC_G(\rho(\Z/p))\rightarrow BG$  by \cite{DZ}.  And this only happens if all the elements of order $p$ are $p$-cohomologically central.
\end{proof}

In the continuous case, there are paradigmatic examples of $BG$ whose cellularization is again a classifying space.

\begin{ejem}
\label{sphere}
If $X=BS^1=K(\Z ,2)$, it is clear comparing pointed mapping spaces that $CW_{B\Z /p}BS^1=B\Z /p$ since $\map_*(B\Z/p, BS^1)$ is homotopically discrete with components $\Hom(\Z/p, S^1)$. Let us now consider $ B S^3$ the classifying space of the $3$-sphere. Lemma \ref{bastaenp} reduces the computation of $CW_{B\Z/p}(BS^3)$ to that of  $CW_{B\Z/p}((BS^3)^{\wedge}_p)$. The mapping space from $B\Z/p$ into $(BS^3)^{\wedge}_p$ has been well studied. If $p=2$, the inclusion of the centre $B\Z/2\rightarrow BS^3$ induces a homotopy equivalence $\map(B\Z/2,B\Z/2)\rightarrow \map(B\Z/2,(BS^3)^{\wedge}_2)$ since $ \map(B\Z/2,(BS^3)^{\wedge}_2)_f\simeq (BC_{S^3}(f))^{\wedge}_2$ (see \cite{DMW87}), and therefore $CW_{B\Z/2}(BS^3)\simeq B\Z/2$. If $p$ is odd, then $(BS^3)^{\wedge}_p\simeq BN(T)^{\wedge}_p$, and this case will be studied in Example \ref{p'-toral}.

%
%There is
%just a non-trivial element of order 2 in $S^3$, which generates
%the center $Z(S^3)$. Hence, by the previous proposition
%$\map( B\Z /2,( B S^3)^{\wedge}_2)$ is discrete, and in
%fact it has only two points corresponding respectively to the
%trivial and the faithful representation of $\Z /2$ in $S^3$. So,
%the map $ B\Z /2\longrightarrow ( B S^3)^{\wedge}_2$ induced by
%the inclusion is a $ B\Z /2$-equivalence, and thus,
%$CW_{ B\Z /2}( B S^3)\simeq B\Z /2$.
\end{ejem}

Sometimes, if we are unable to describe $CW_{B\Z /p}BG$, we can at least identify it with another classifying space at a prime.

\begin{ejem}
\label{O(2)}
Let $BO(2)$ be the classifying space of the orthogonal group $O(2)$. There is a mod $2$ equivalence $BD_{2^{\infty}}\rightarrow BO(2)$ where $D_{2^\infty}=colim_n D_{2^n}$. Moreover $BD_{2^\infty}$ is $B\Z/2$-cellular by \cite[Example 5.1]{Ramon}. Since $\pi_1(BO(2))=\Z/2$ is generated by an element of order $2$ which lifts to $BO(2)$, we are in the situation of Remark \ref{remark-p-toral}. This will be used in particular in Proposition \ref{BSO}.
% by Proposition \ref{cw+comp} and Remark \ref{remark-p-toral}, the map $BD_{2^\infty}\rightarrow CW_{B\Z/2}(BO(2))$ is a mod $2$ equivalence.
\end{ejem}

We devote the remaining of the section to study some families of Lie groups which show different and interesting features in this context. We begin with extensions of elementary abelian $p$-groups by a finite group of order prime to $p$, which provide an example in which Proposition \ref{cw+comp} does not hold. Compare with \cite{FS07}.  We start with a situation which deals with fibrations.

\begin{prop}\label{nonmodular-fibration}
Let $F\rightarrow E\rightarrow B$ be a fibration of $p$-good connected spaces such that $F$ is $B\Z/p$-cellular, $B$ is $B\Z/p$-null and, $B^\wedge_p$ is $\Sigma B\Z/p$-null. Assume that $[B\Z/p,E]\rightarrow [B\Z/p,E^\wedge_p]$ is exhaustive and $\pi_1(F)\rightarrow \pi_1(E^\wedge_p)$ is an epimorphism. Then $(CW_{B\Z/p}(E^\wedge_p))^\wedge_p$ is the homotopy fiber of $E^\wedge_p\rightarrow B^\wedge_p$.
\end{prop}

\begin{proof} First of all, note that since $B$ is $B\Z/p$-null, then $F\rightarrow E$ is a $B\Z/p$-equivalence, and thus $F\simeq CW_{B\Z/p}(E)$.
To compute the cellularization of $E^{\wedge}_p$ we proceed by applying Chach\'olski's strategy (Theorem \ref{Chacho}, see also \cite[Section 7]{Chacholski96} for the slightly general formulation we use here). Consider the following diagram of horizontal cofibrations,
$$\xymatrix{
F \ar[d] \ar[r]  & E \ar[d] \ar[r] & C \ar[d]^g \\
F \ar[r]^{i}  & E^{\wedge}_p  \ar[r] & D
}$$ where $D$ is $1$-connected since $i$ is an epimorphism on fundamental groups. Since $E$ is $p$-good, $g$ induces an homotopy equivalence $C^\wedge_p\simeq D^\wedge_p$, and therefore $C^\wedge_p$ is also $1$-connected.

Zabrodsky's Lemma (see \cite[Prop 3.4]{MR97i:55028}) applied to the fibration $F\rightarrow E\rightarrow B$ and the composite map $E\rightarrow C\ra P_{\Sigma B\Z/p}(C)$ implies that there is a map $B \rightarrow P_{\Sigma B\Z/p}(C)$ which fits in a diagram of fibrations:
$$\xymatrix{
F \ar[d] \ar[r]  & E \ar[d] \ar[r] & B \ar[d] \\
CW_{B\Z/p}(E) \ar[r]  & E  \ar[r] & P_{\Sigma B\Z/p}C.
}$$ where the first vertical map is a homotopy equivalence, $CW_{B\Z/p}(E)\simeq F$.
The long exact sequence for homotopy groups shows that the last vertical arrow is also a homotopy equivalence.
Now consider the diagram of fibre sequences
$$\xymatrix{
F \ar[d] \ar[r]  & E \ar[d] \ar[r] &P_{\Sigma B\Z/p}C\simeq B \ar[d] \\
CW_{B\Z/p}(E^{\wedge}_p) \ar[r]  & E^{\wedge}_p  \ar[r] & P_{\Sigma B\Z/p}D.
}$$
The spaces  $P_{\Sigma B\Z/p}(C)\simeq B$ and $P_{\Sigma B\Z/p}(C^{\wedge}_p)$ are $p$-good spaces (note that $P_{\Sigma B\Z/p}(C^{\wedge}_p)$ is $1$-connected by Lemma \ref{1-connectednullification}) and Miller's theorem apply to show that  $P_{\Sigma B\Z/p}(C^{\wedge}_p)^{\wedge}_p$ is $\Sigma B\Z/p$-null. Also $P_{\Sigma B\Z/p}(C)^{\wedge}_p\simeq B^{\wedge}_p$ is $\Sigma B\Z/p$-null by hypothesis. Applying Lemma \ref{null+comp} and the proof of Corollary \ref{null+modp+map}, we obtain that the composite $P_{\Sigma B\Z/p}(C)\simeq B \rightarrow P_{\Sigma B\Z/p}(D)$ is a mod $p$ equivalence, and therefore we conclude that  the $p$-completion $(CW_{B\Z/p}(E^{\wedge}_p))^\wedge_p$ is the homotopy fiber of $E^\wedge_p \rightarrow B^{\wedge}_p$ is a mod $p$ equivalence by $p$-completion of the fibration $CW_{B\Z/p}(E^\wedge_p)\rightarrow E^\wedge_p\rightarrow P_{\Sigma B\Z/p}(D)$.
\end{proof}

\begin{cor}
Let $F\rightarrow E\rightarrow B$ be a fibration of connected spaces such that $F$ is $B\Z/p$-cellular, $B$ is $B\Z/p$-null and $B^\wedge_p\simeq *$. Assume that $[B\Z/p,E]\rightarrow [B\Z/p,E^\wedge_p]$ is exhaustive, $\pi_1(F)\rightarrow \pi_1(E^\wedge_p)$ is an epimorphism and $\pi_i(E)$ are finite groups for all $i\geq 1$. Then $E^\wedge_p$ is $B\Z/p$-cellular.
\end{cor}

\begin{proof}
By Proposition \ref{nonmodular-fibration}, we know that $(CW_{B\Z/p}(E^\wedge_p))^\wedge_p\simeq E^\wedge_p$. We will prove that $CW_{B\Z/p}(E^\wedge_p)$ is $p$-complete. Since  $\pi_i(E)$ are finite groups for all $i\geq 1$, $\pi_i(E^\wedge_p)$ are all finite $p$-groups and $E^\wedge_p$ is nilpotent (\cite[VII.4.3]{BK}). Therefore $CW_{B\Z/p}(E^\wedge_p)$ is nilpotent by Lemma \ref{nilpotent-cellularization}. A Sullivan's  arithmetic square argument shows that $CW_{B\Z/p}(E^\wedge_p)$ is $p$-complete since $((CW_{B\Z/p}(E^\wedge_p))^\wedge_p)_\Q\simeq (E^\wedge_p)_\Q\simeq *$.
\end{proof}

\begin{ejem}\label{rorder}
Let $G$ be a finite group which is an extension $H\rightarrow G\rightarrow W$ where $BH$ is $B\Z/p$-cellular and $(|W|,p)=1$. Then $CW_{B\Z/p}(BG)\simeq BH$ and $BG^\wedge_p$ is $B\Z/p$-cellular by the previous result. Note that $G$ does not need to be generated by elements of order $p$; compare with \cite[Section 4]{FS07}. Other examples are provided by nilpotent Postnikov pieces whose fundamental group is of order prime to $p$ and the $1$-connected cover is $p$-torsion.
\end{ejem}

\begin{ejem}\label{p'-toral}
Let $N$ be an extension of a finite group of order prime to $p$ with a torus, that is, we have a fibration $BT\rightarrow BN \rightarrow BW$ where $T\cong (S^1)^n$ and $(|W|,p)=1$. From this fibration we see that $CW_{B\Z/p}(BN)\simeq CW_{B\Z/p}(BT)\simeq BV$ where $V\cong (\Z/p)^n$, as $BW$ is $B\Z/p$-null and $BT\rightarrow BN$ is a $B\Z/p$-equivalence.

Next we compute the cellularization of $(BN)^\wedge_p$. First, by \cite[Prop. 7.5]{Broto-Kitchloo}, there is a bijection $[B\Z/p,BN]\cong[B\Z/p,BN^\wedge_p]$.  Consider the following diagram of horizontal cofibrations,
$$\xymatrix{
BV \ar[d] \ar[r]  & BN \ar[d] \ar[r] & C \ar[d]^g \\
BV \ar[r]^{i}  & BN^{\wedge}_p  \ar[r] & D.
}$$ where $D$ is $1$-connected since $BN^{\wedge}_p$ is also $1$-connected. Therefore $P_{\Sigma B\Z/p}(D)$ is also $1$-connected by Lemma \ref{1-connectednullification}. Since $\pi_1(C)$ is finite, $C$ is $p$-good. Moreover, $g$ is a mod $p$ equivalence, therefore $C^\wedge_p$ is $1$-connected.  Now consider the following diagram of fibrations:
$$\xymatrix{
BV \ar[d] \ar[r]  & BN \ar[d] \ar[r] &P_{\Sigma B\Z/p}C \ar[d] \\
CW_{B\Z/p}(BN^{\wedge}_p) \ar[r]^{i}  & BN^{\wedge}_p  \ar[r] & P_{\Sigma B\Z/p}D.
}$$
We will show that $P_{\Sigma B\Z/p}(g)\colon P_{\Sigma B\Z/p}C\rightarrow P_{\Sigma B\Z/p}D$ is a mod $p$ equivalence. Since $g$ is a mod $p$ equivalence, and also $P_{\Sigma B\Z/p}(\eta_D)\colon P_{\Sigma B\Z/p}(D)\rightarrow P_{\Sigma B\Z/p}(D^\wedge_p)$ by Corollary \ref{null+modp}, we only need to prove that  $P_{\Sigma B\Z/p}(\eta_C)\colon P_{\Sigma B\Z/p}(C)\rightarrow P_{\Sigma B\Z/p}(C^\wedge_p)$ is also a mod $p$ equivalence by checking that $C$ satisfies the hypothesis of Lemma \ref{null+comp}.

Consider the following diagram of fibrations
$$\xymatrix{
T/V \ar[r] \ar[d] & BV \ar[r] \ar[d]  & BT \ar[d] \\
E(T/V) \ar[r] \ar[d]  & BN \ar[r] \ar[d] & BN \ar[d] \\
B(T/V) \ar[r] & P_{\Sigma B\Z/p}(C) \ar[r]^f & BW.
}$$
where $f$ exists by Zabrodsky's Lemma (see \cite[Prop 3.4]{MR97i:55028}) applied to the fibration $BV\rightarrow BN \rightarrow P_{\Sigma B\Z/p}(C)$ and the map $BN\rightarrow BW$. It  implies that there is a map $P_{\Sigma B\Z/p}(C) \rightarrow BG$ which fits in a diagram of fibrations. The bottom fibration shows that $P_{\Sigma B\Z/p}(C)$ is homotopy equivalent to the classifying space of a compact Lie group whose fundamental group is $G$.

Now we check that $C$ satisfies the hypothesis of Lemma \ref{null+comp}. First, $C$ and $P_{\Sigma B\Z/p}(C)$ are $p$-good since they have finite fundamental groups (\cite[VII.5.1]{BK}), $P_{\Sigma B\Z/p}(C^\wedge_p)$ is $1$-connected and therefore it is also $p$-good. It remains to check that $P_{\Sigma B\Z/p}(C)^\wedge_p$ and  $P_{\Sigma B\Z/p}(C^\wedge_p)^\wedge_p$ are $\Sigma B\Z/p$-null spaces. Since $P_{\Sigma B\Z/p}(C)$ is homotopy equivalent to the classifying space of a compact Lie group, its $p$-completion is $\Sigma B\Z/p$-null (see e.g. \cite[Prop 7.5]{Broto-Kitchloo}). Finally $P_{\Sigma B\Z/p}(C^\wedge_p)^\wedge_p$ is also $\Sigma B\Z/p$-null since $P_{\Sigma B\Z/p}(C^\wedge_p)$ is $1$-connected by Theorem \ref{bastaenp-Miller}.

Summarizing, $(CW_{B\Z/p}(BN^\wedge_p))^\wedge_p$ is the homotopy fiber of $BN^\wedge_p\rightarrow BK^\wedge_p$ where $K$ is an extension of $W$ by $T/V$ and $V$ is the maximal elementary abelian $p$-subgroup in the torus $T$.
\end{ejem}

% and will emphasize the importance of a distinguished strongly closed subgroup of the Sylow, which already was crucial in the description of the cellularization of finite groups \cite[Section 4]{FS07}.
Our next example concerns $p$-toral groups. Recall that a $p$-toral group is an extension of a torus by a finite $p$-group.
A $p$-compact toral group is an extension of a $p$-compact torus by a finite $p$-group, and a discrete $p$-toral group is a group $P$ with normal subgroup $T$ such that $T$ is isomorphic to a finite product of copies of $\Z/p^\infty$ and $P/T$ is a finite $p$-group.

Since $CW_{B\Z/p}(BT^{\wedge}_p)\simeq CW_{B\Z/p}(BT)$ by Lemma \ref{bastaenp} and  $CW_{B\Z/p}(BT)\simeq BV$ where $V$ is the subgroup of elements of order $p$, the following
is also true for $p$-compact toral groups.

\begin{ejem}\label{cw-p-toral}
Let $P$ be a $p$-toral group with group of components $\pi$. First of all, by \cite[Proposition 2.1]{CCS2}, we can assume that $\pi$ is a finite $p$-group generated by elements of order $p$ which lift to $BP$. By Proposition \ref{cw+comp} and Remark \ref{remark-p-toral}, there is a mod $p$ equivalence $CW_{B\Z/p}(BP)\rightarrow CW_{B\Z/p}(BP^\wedge_p)$.
Dwyer and Wilkerson show in \cite{DW-annals} that there exists a discrete $p$-toral group $P_\infty$ such that $BP_\infty \rightarrow BP$ is a mod $p$ equivalence. We are reduced then to study the cellularization of discrete $p$-toral groups.

Following \cite[Section 4]{FS07}, we consider $\Omega_1(P_\infty)$, the subgroup generated by the elements of order $p$. Since a subgroup of a $p$-toral discrete group is also a $p$-toral discrete group and the map $B\Omega_1(P_\infty)\rightarrow BP_\infty$ is a $B\Z/p$-cellular equivalence (note that  $\map_*(B\Z/p,BP_\infty)\simeq Hom(\Z/p, P_\infty)$), we can assume that $P_\infty$ is generated by elements of order $p$. For any $p$-discrete toral group there is an increasing sequence  $P_0\leq P_1\leq \cdots $ such that $P_\infty=\cup P_n$. Take a countable set of generators of order $p$ for $P_\infty$, $\{g_i|i=1,\ldots,n\}$; then the subgroups $Q_n=\langle g_1,\ldots, g_n\rangle$ satisfy that $P_\infty=\cup Q_n$ and each $Q_n$ is a finite $p$-group generated by elements of order $p$, so by \cite[Prop 4.14, Prop 4.8]{Ramon}, $BQ_n$ is $B\Z/p$-cellular and therefore $BP_\infty$ is so.

Finally the space $B\Omega_1(P_\infty)$ is $B\Z/p$-cellular, so it remains to check that $B\Omega_1(P_\infty)\rightarrow CW_{B\Z/p}(BP^\wedge_p)$ is a mod $p$ equivalence. Let $C_{BP_\infty}$ and $C_{BP^\wedge_p}$ be the corresponding Chach-\'{o}lski's cofibres.  Zabrodsky's Lemma (see \cite[Prop 3.4]{MR97i:55028}) applied to the fibration $B\Omega_1(P_\infty)\rightarrow BP_\infty \rightarrow P_{\Sigma B\Z/p}(C_{BP_\infty})$ and the map $BP_\infty\rightarrow B(P_\infty/\Omega_1(P_\infty))$ shows that there is a homotopy equivalence $ P_{\Sigma B\Z/p}(C_{BP_\infty})\simeq B(P_\infty/\Omega_1(P_\infty))$. In particular, $C_{BP\infty}$ satisfies the hypothesis of Lemma \ref{null+comp}. Moreover, $C_{BP^\wedge_p}$ is $1$-connected.

The map $g\colon C_{BP\infty}\rightarrow C_{BP^\wedge_p}$ is a mod $p$ equivalence and, by Remark \ref{remark+null+map} and Corollary \ref{null+modp+map}, $P_{\Sigma B\Z/p}(g)$ is a mod $p$ equivalence. Finally, Proposition \ref{cw+comp} combined with the previous results, show that $B\Omega_1(P_\infty)\rightarrow CW_{B\Z/p}(BP^\wedge_p)$ is a mod $p$ equivalence.

In particular, from Example \ref{O(2)} we obtain that there are mod $2$ equivalences $BD_{2^\infty}\rightarrow CW_{B\Z/2}(BO(2))\rightarrow CW_{B\Z/2}(BO(2)^\wedge_2)$, and hence a chain of homotopy equivalences $CW_{B\Z/2}(BO(2))^\wedge_2\simeq CW_{B\Z/2}(BO(2)^\wedge_2)^\wedge_2\simeq BO(2)^\wedge_2$.
%the map $B\Omega_1(P_\infty)\rightarrow CW_{B\Z/p}(BP)$ is a mod $p$ equivalence.
\end{ejem}

We finish the section with a last example in which we can observe a completely different pattern, and where the cellularization is obtained by combining in an adequate way some nice push-out decompositions.

\begin{prop}
\label{BSO}
The $\emph{B}\Z /2$-cellularization of ${B}SO(3)$ fits in a
fibration $$({CW}_{\emph{B}\Z
/2}{B}SO(3))^{\wedge}_2 \ra{B}SO(3)^{\wedge}_2 \ra ({B}SO(3)^{\wedge}_2)_\Q.$$
\end{prop}

\begin{proof}
Since $SO(3)$ is connected, by  Lemma \ref{bastaenp} the $p$-completion induces a
homotopy equivalence ${CW}_{ B\Z /p}BSO(3)\simeq{CW}_{ B\Z /p}(BSO(3)^{\wedge}_p)$. According to \cite[Cor 4.2]{DMW87},
$BSO(3)$ is equivalent at the prime 2 to the pushout $X$ of the following
diagram:
$$
\xymatrix{BD_8 \ar[r]^{f_2} \ar[d]^{f_1} & BO(2)^{\wedge}_2 \ar[d]^g \\
(B\Sigma_4)^{\wedge}_2 \ar[r] & X,\\ }
$$ where $f_1$ is induced by inclusion of the 2-Sylow subgroup, and $f_2$ is
given by the map of extensions
$$
\xymatrix{ \Z /4 \ar[d] \ar[r] & D_8 \ar[d]^{f_2} \ar[r] & \Z /2 \ar[d]
\\
SO(2) \ar[r] & O(2) \ar[r] & \Z /2. \\ }
$$

Our strategy will be to cellularize the previous diagram, and
compare the respective pushouts. Now, recall that $ B D_8$ is
$ B\Z /2$-cellular (\cite[4.14]{Ramon}) and moreover the
2-completion of $ B\Sigma_4$ is so (\cite[Thm 4.4]{FS07}).
On the other hand, $D_{2^\infty}$ is a $2$-discrete approximation of $O(2)$ -i.e.$BD_{2^\infty}\ra BO(2)$ is a mod $2$ equivalence-, so the previous Example \ref{cw-p-toral} implies
$BD_{2^\infty}\rightarrow CW_{B\Z/2}(BO(2))$ is a mod $2$ equivalence. Moreover, by Proposition \ref{cw+comp} and Remark \ref{remark-p-toral}, there is also a mod $2$ equivalence $CW_{B\Z/2}(BO(2))\rightarrow CW_{B\Z/2}(BO(2)^\wedge_2)$. So, we can consider another pushout diagram by applying the functor $CW_{B\Z/2}$ to the previous one,
$$
\xymatrix{ B D_8 \ar[r] \ar[d]^{f_1} &  CW_{B\Z/2}(BO(2)^{\wedge}_2) \ar[d]^h \\
( B\Sigma_4)^{\wedge}_2 \ar[r] & Y\\ }
$$

There exists a map $g\colon Y\ra
X$ induced by the augmentation map from the cellularization which is a mod $2$ equivalence since $CW_{B\Z/2}(BO(2)^{\wedge}_2)\rightarrow BO(2)^{\wedge}_2$ is so.

 Now we attempt to compute the $B\Z/2$-cellularization of $X^{\wedge}_2$ by using the cofibre of the map $k\colon Y\rightarrow X^\wedge_2$.  In order to do this, a result of Chach\'{o}lski \cite[Thm 20.3]{Chacholski96} together with \cite[Thm 1.1]{FS07} tells us that we need to check that $[B\Z/2, Y]\rightarrow [B\Z/2,Y^{\wedge}_2]\cong[B\Z/2,X^{\wedge}_2]$ is exhaustive and $Y$ is $B\Z/2$-cellular. $Y$ is $B\Z/2$-cellular since it is a pushout of  $B\Z/2$-cellular spaces. It remains to check that $[B\Z/2, Y]\rightarrow [B\Z/2,Y^{\wedge}_2]$ is exhaustive.

Let $\mathcal P$ be the category $1\leftarrow 0 \rightarrow 2$ describing a pushout diagram, and let $F\colon \mathcal P \rightarrow Top$ be the functor describing the pushout for $Y$, that is, $F(1)=( B\Sigma_4)^{\wedge}_2$, $F(0)= B D_8$ and $F(2)=CW_{B\Z/2}(BO(2)^{\wedge}_2)$ with the corresponding morphisms. There is a commutative diagram of sets
$$
\xymatrix{
\limdir[B\Z/2,F] \ar[r] \ar[d]^{(\eta_F)_*} & [B\Z/2,Y] \ar[d]^{\eta_*} \\
\limdir [B\Z/2,F^\wedge_2] \ar[r] & [B\Z/2,Y^\wedge_2] }
$$
where the vertical maps are induced by $2$-completion of the target. Since the spaces $\map(B\Z/2,F^\wedge_2)$ are $2$-complete (see \cite[Proposition 7.5]{Broto-Kitchloo}), by \cite[Lemma 4.2]{BLO2} the bottom horizontal map is a bijection. To prove that $\eta_*$ is exhaustive, it is enough to show that $(\eta_F)_*$ is so. But then, looking at the diagram, it reduces to check that $[B\Z/2, CW_{B\Z/2}(BO(2)^\wedge_2)]\rightarrow [B\Z/2, CW_{B\Z/2}(BO(2)^\wedge_2)^\wedge_2]\cong[B\Z/2,BO(2)^\wedge_2]$ is exhaustive  (see Example \ref{cw-p-toral}) and this follows from Corollary \ref{unpointed}.

Let $C$ be the cofibre of $k$. We know that $C$ is mod $2$ acyclic and $1$-connected. Now if $q$ is an odd prime, since $Y$ is $B\Z/2$-cellular, $Y$ is mod $q$ acyclic and $C^{\wedge}_q\simeq (BSO(3)^\wedge_2)^{\wedge}_q$ is contractible. Finally $C_\Q\simeq (BSO(3)^\wedge_2)_\Q$. Then, by a Sullivan arithmetic square argument, $C\simeq (BSO(3)^\wedge_2)_\Q$ which is, in turn, $B\Z/2$-null. In particular $C$ is $\Sigma B\Z/2$-null. Therefore, the fibration of the theorem follows from Chach\'olski's fibration describing the cellularization (Theorem \ref{Chacho}).
\end{proof}

\begin{remark}
Note that if $p$ is an odd prime, then $BSO(3)^{\wedge}_p\simeq BN(T)^{\wedge}_p$, where $N(T)$ is the normalizer of the maximal torus, and we analyzed this case in Example \ref{p'-toral}.
\end{remark}

It seems natural to ask  if the problem of computing $CW_{B\Z/p}(BG)$ for any compact Lie group $G$ is accessible at this point. A strategy was developed for finite groups in \cite{FFcw}, based in the description of the strongly closed subgroups of $G$, which are classified. Recent research has remarked the role of the strongly closed subgroups of discrete $p$-toral groups in the homotopy theory of compact Lie groups and, more generally,  $p$-local compact groups \cite{Alex}, but to our knowledge there is no available classification of these objects. On the other hand, the nontrivial rational homotopy of $BG$ seems an important obstacle to generalize the arithmetic square arguments of the strategy. We plan to undertake these issues in subsequent work, and, in particular, an intriguing question which arises in a natural way from the last example:
\medskip

\textbf{Question}: For which class of classifying spaces of compact Lie groups (or spaces in general) is the $B\Z /p$-cellularization equivalent to the
homotopy fibre of the rationalization, up to $p$-completion?

%%%%%%%%%%%%%%%%%%%%%%%%%%

{\bf Acknowledgements.} We would like to thank Carles Broto and J\'{e}r\^{o}me Scherer for interesting conversations on this subject.
%%%%%%%%%%%%%%%%%%%%%%%%%%%%%%%%%%%%%%%%%

%\nocite{*}

\end{document}